\documentclass{amsart}

\usepackage{graphicx}
\usepackage{bbm,amsmath,amsfonts,amssymb}
\usepackage[mathscr]{eucal}

\usepackage{natbib}

\newtheorem{theorem}{Theorem}[section]
\newtheorem{lemma}[theorem]{Lemma}

\newtheorem{proposition}[theorem]{Proposition}
\newtheorem{definition}[theorem]{Definition}
\newtheorem{example}[theorem]{Example}
\newtheorem{remark}[theorem]{Remark}

\newcounter{figures}[section]

\RequirePackage[colorlinks,allcolors =black]{hyperref}

\def\bN{{\mathbb N}}

\def\bR{{\mathbb R}}
\def\bL{{\mathbb L}}

\def\bS{{\mathbb S}}

\def\bB{{\mathbb B}}

\def\cB{\mathcal{B}}
\def\cC{\mathcal{C}}
\def\cD{\mathcal{D}}

\def\cF{\mathcal{F}}
\def\cG{\mathcal{G}}
\def\cH{\mathcal{H}}

\def\cK{\mathcal{K}}

\def\cM{\mathcal{M}}

\def\cX{\mathcal{X}}

\def\supp{{\rm{supp}\, }}

\def\R{\mathbb{R}}

\def\dd{{\delta}}

\def\mm{\tau}

\def\ddelta{h}

\begin{document}

\title[Density estimation on metric spaces]
{Pointwise density estimation on metric spaces and applications in seismology}

\author{G. Cleanthous}
\address{Department of Mathematics and Statistics,
National University of Ireland, Maynooth}
\email{galatia.cleanthous@mu.ie}
\author{A. G. Georgiadis}
\address{School of Computer Science and Statistics,
Trinity College of Dublin}
\email{georgiaa@tcd.ie}
\author{P. A. White}
\address{Department of Statistics,
Brigham Young University}
\email{pwhite@stat.byu.edu}



\begin{abstract}
We are studying the problem of estimating density in a wide range of metric spaces, including the Euclidean space, the sphere, the ball, and various Riemannian manifolds. Our framework involves a metric space with a doubling measure and a self-adjoint operator, whose heat kernel exhibits Gaussian behaviour. We begin by reviewing the construction of kernel density estimators and the related background information. As a novel result, we present a pointwise kernel density estimation for probability density functions that belong to general H\"{o}lder spaces. The study is accompanied by an application in Seismology. Precisely, we analyze a globally-indexed dataset of earthquake occurrence and compare the out-of-sample performance of several approximated kernel density estimators indexed on the sphere.
\end{abstract}

\date{February, 2023}

\maketitle

\section{Introduction}\label{Introduction}
\setcounter{equation}{0}

Today, technology has equipped science with a massive amount of data that requires rigorous analysis. In astronomy, data can come from missions to other planets, telescopes observing distant parts of the universe, or programs studying Cosmic Microwave Background Radiation. In climatology and environmental science, sensors provide data on the atmosphere. Medicine uses scans to track the growth of tumors and monitor their development, while embryology uses data to track the growth and ensure the health of developing humans. Essentially all scientific fields now heavily rely on data.

The complexity and form of the data reflect their nature. In the examples mentioned above, the data can be represented by geometric structures that capture their form and dynamics. A dataset should be understood as independent realizations of a random variable (rv) $X$. Such a rv lives in some domain according to its nature. For instance when $X$ represents the locations on some planet, then $X$ lives on the sphere $\bS^2$ of the Euclidean space $\mathbb{R}^3$. The same is true of CMB radiation. For geological data in the interior of Earth, or another celestial body, the domain of study may be the ball $\bB^3$. Similarly, in the field of medicine, the domain of definition of $X$ can become much more complicated geometrically, and as a result, the general target domain becomes an abstract metric space $\cM$.

Let $X$ be a rv distributed on a metric measure space $\cM$ and let $f=f_{X}$ be its unknown probability density function (pdf). Density estimation, estimating a pdf from data $X_1,\dots,X_n$, represents an important problem in Statistics. To this end we need to construct a \textit{density estimator}, which is an object of the form $\hat{f}_{n}(X_1,\dots,X_n;x)$, where $\hat{f}_n:\cM^{n}\times\cM\rightarrow\bR$ a measurable function. A famous method for obtaining such an estimator is by the so-called ``kernel density estimators". 

Nonparametric Statistics approaches the problem of density estimation by constructing appropriate kernel density estimators, which can approximate any density with membership is certain regularity spaces. Historically, these methods were pioneered by \cite{Ros}, \cite{Parzen} and  \cite{BH}. The first books on the topic include \cite{silverman} and \cite{HGPT}, while today the book \cite{Tsybakov} is considered one of the main reference points. For an indicative list of contributions we refer to \cite{BKMP,Baraud,BM,BS,Birge,DG,DL,DL2,DJKP,EY,GL,GL2,GL3,GLne1,GLne2,HWC,HI,IK,KePT,KeLP,KLP2,Massart,Pel1,Pel2,Rig,RigT,ST}.

\vspace{0.2cm}

Here we study kernel density estimators on \textit{metric measure spaces} under very broad assumptions. The setting  we will work covers simultaneously the classical cases of Euclidean space $\bR^d$, the sphere $\mathbb{S}^{d}$, the ball $\mathbb{B}^d$ and many more significant examples of independent interest. Furthermore it contains more sophisticated geometric settings like manifolds and Lie groups. On the other hand, some techniques originated from spectral theory will simplify and unify several aspects of the approach. We shall operate in the setting put forward in \cite{CKP}, which we describe next in a simplified form:


\vspace{0.2cm}

I. We assume that $(\cM,\rho,\mu)$ is a metric measure space such that
$(\cM, \rho)$ is locally compact with distance $\rho(\cdot, \cdot)$
and $\mu$ is a~positive Radon measure
satisfying:

\vspace{0.2cm}

\noindent
(i) {\em Ahlfors regularity:} There exist constants $c_1\ge1$ and $d>0$ such that
\begin{equation}\label{Ahlfors}
c_1^{-1}r^d\le |B(x,r)| \le c_1 r^{d}
\quad\hbox{for every $x \in \cM$ and $r>0$,}
\end{equation}
where $|B(x,r)|$ is the volume of the open ball $B(x,r):=\{y\in \cM:\rho(x,y)<r\}$ centred at $x$ of radius $r$.

The number $d$ is the so-called {\em Ahlfors dimension of the space}.

\vspace{0.2cm}


II. The second assumption is that there exists
an essentially self-adjoint non-negative operator $L$ on $\bL^2(\cM, d\mu)$,
mapping real-valued to real-valued functions,
such that the associated semigroup (more details in \S \ref{kdemetric}) $P_t=e^{-tL}$, $t>0$, consists of integral operators with
(heat) kernel $p_t(x,y)$ obeying the conditions:

\smallskip

\noindent
(ii) {\em Gaussian localization:} There exist constants $c_2,c_3>0$ such that
\begin{equation}\label{Gauss-local}
|p_t(x,y)|
\le c_2 t^{-d/2}\exp\Big\{-c_3\frac{\rho^2(x,y)}t\Big\}
\quad\hbox{for every} \;\;x,y\in \cM,\,t>0.
\end{equation}

\noindent
(iii) {\em H\"{o}lder continuity:} There exists a constant $\alpha>0$ such that
\begin{equation}\label{lip}
\big|  p_t(x,y) - p_t(x,y')  \big|
\le c_2\Big(\frac{\rho(y,y')}{\sqrt t}\Big)^\alpha
t^{-d/2}\exp\Big\{-c_3\frac{\rho^2(x,y)}t\Big\}
\end{equation}

for every $x, y, y'\in \cM$ such that $\rho(y,y')\le \sqrt{t}$ and $t>0$.

\smallskip

\noindent
(iv) {\em Markov property:}
\begin{equation}\label{hol3}
\int_{\cM} p_t(x,y) d\mu(y)= 1
\quad\hbox{for every $x\in \cM$ and $t >0$.}
\end{equation}

\vspace{0.2cm}

This setting we study generalizes (by default) the Euclidean space. Moreover, it contains spaces like the sphere, the ball, the interval, cubes/rectangles, the simplex, Riemannian manifolds with non-negative Ricci curvature and more, each equipped with their natural metrics and measures associated with Laplace or Laplace-Beltrami operators. For more examples we refer the reader to \cite{CKP,GN,KPX,KP}. 

Some first contributions in Statistics in this generality can be found in \cite{CaKPi,CGKPP,CGP,KOPP}, while there is a large number of open problems in front of the community. 
\vspace{0.2cm}

The \textit{aim} of the present study is threefold:

$(\alpha)$ To \textit{review} the setting and the construction of kernel density estimators together with the corresponding theoretical background, which is demanding, on the broad framework under study; Section \ref{kdemetric}.

$(\beta)$ As \textit{novel results}, we obtain optimal pointwise density estimation on H\"{o}lder spaces; see Sections \ref{Sec:MSE} and \ref{Sec:Doubling} and we shed light in the assumptions and methods.

$(\gamma)$ As an \textit{application}, we perform a data-analysis of earthquakes; Section \ref{S:data}, using our kernel density estimators. Precisely we compare the out-of-sample performance of several approximated kernel density estimators and we plot the heat map of the estimated density using the selected model.

\vspace{0.2cm}

Remarks and Examples are placed in several points of the manuscript for highlighting notions and ideas.
The new results are contained in Section \ref{Sec:MSE} and under more general assumptions in Section \ref{Sec:Doubling} and are accompanied with remarks that could be used for future studies.

Section \ref{S:data} is dedicated to the data analysis of \textit{earthquakes}. In this Section we apply the theoretical results of the paper and show how one can use these approaches with occurrence data on the Earth. The data used in this analysis are freely available through the United States Geological Survey website \url{https://earthquake.usgs.gov/earthquakes/search/}.
\vspace{0.2cm}

\noindent
{\em Notation}:
Throughout positive constants will be denoted by $c$, and will be allowed to vary
at every occurrence. The dependence of a constant to the geometric structure constants $c_1,c_2,c_3,\alpha$ and $d$ will not be stated, but the dependence to a parameter $q$, will be stated as $c_q$. We denote by $\bN,\;\bR,\;\bR_+$ the sets of positive integers, real numbers and non-negative real numbers respectively. If $\tau\in\bN$, the class of differentiable functions on $\bR_+$ with continuous derivatives up to order $\tau$ will be stated as $\cC^{\tau}(\bR_+)$. For $s>0$, we will denote by $\lfloor s\rfloor$ the greatest integer that is strictly less than $s$ and by $\lceil s\rceil$ the smaller integer strictly larger than $s$. 

\section{Density estimation on metric spaces associated with operators: \\ A review}\label{kdemetric}


The first part of our study consists of a review of density estimation on metric spaces associated with operators. One of the milestones is to construct kernels. We expand here the methods used in \cite{CGKPP,CGP} inspired by the corresponding machinery built in \cite{CKP} based on the powerful Spectral Theory.

\subsection{Functional calculus}\label{subsec:func-calc}
We start by some fundamental notions of Spectral Theory providing a minimum background of this wide scientific field; the reader is further referred to
\cite{Prugov,RS,Yosida}. 
\vspace{0.2cm}

Recall that $L$ is assumed to be a non-negative self-adjoint operator
that maps real-valued to real-valued functions. Then \cite[Section 5]{Prugov} $L$ admits a \textit{unique spectral measure} $E$; that is a projector-valued mapping as follows: 

Denote by $\cB$ the Borel $\sigma$-algebra on $\bR$. For every $S\in\cB$, we correspond an orthogonal projection $E(S):\bL^2(\cM,d\mu)\rightarrow \bL^2(\cM,d\mu)$  such that:

(i) $E(\bR)=I$ (the identity operator on $\bL^2(\cM,d\mu)$).

(ii) For every sequence of disjoint Borel sets $\{S_n\}_{n\in\bN}\subset\cB$
\begin{equation}
E(S)=\sum_{n=1}^{\infty}E(S_n),\quad\text{where}\;\; S:=\bigcup_{n=1}^{\infty}S_{n},
\end{equation}
in the strong $\bL^2(\cM,d\mu)$ sense; i.e. for every $f\in\bL^2(\cM,d\mu)$,
\begin{equation}
\Big\|\Big(E(S)-\sum_{n=1}^{N}E(S_n)\Big)f\Big\|_2\xrightarrow{N\rightarrow\infty}0.
\end{equation}

Thanks to (ii), for every $f,g\in\bL^2(\cM,d\mu)$ the set-function
\begin{equation}
\nu_{f,g}(S):=\langle E(S)f,g\rangle,\quad\text{for every }\;S\in\cB,
\end{equation}
is a complex measure on $(\bR,\cB)$. 

Moreover for every $f\in\bL^2(\cM,d\mu)$ the set-function
\begin{align*}
\nu_{f}(S)&:=\nu_{f,f}(S)=\langle E(S)f,f\rangle
\\
&=\langle E(S)^2 f,f\rangle=\langle E(S)f,E(S)f\rangle=\|E(S)f\|_2^2,\quad S\in\cB
\end{align*}
is a measure on $(\bR,\cB)$, which is finite and precisely $\nu_{f}(\bR)=\|f\|_2^2<\infty$.
\vspace{0.2cm}

The study can be slightly simplified by the following projector-valued function
\begin{equation}
\bR\ni\lambda\mapsto E_{\lambda}:=E(I_\lambda),\quad I_{\lambda}:=(-\infty,\lambda]
\end{equation}
which is referred as the \textit{spectral resolution of} $L$. Moreover for every $f\in\bL^2(\cM,d\mu)$ and every $\lambda\in\bR$, we have $\nu_{f}(\lambda):=\nu_{f}(I_\lambda)=\langle E_{\lambda}f,f\rangle=\|E_{\lambda}f\|_2^2$. 

Given further that $L$ is assumed non-negative, by \cite[Theorem 6.3]{Prugov}, the domain $\text{Dom}(L)$ of $L$ consists of all functions $f\in\bL^2(\cM,d\mu)$ such that
\begin{equation}
\int_{0}^{\infty}\lambda^2 d\nu_{f}(\lambda)=\int_{0}^{\infty}\lambda^2 d\langle E_{\lambda}f,f\rangle<\infty.
\end{equation}
Moreover for every $f\in\text{Dom}(L)$ and $g\in\bL^2(\cM,d\mu)$
\begin{equation}
\langle Lf,g\rangle=\int_{0}^{\infty} \lambda d\langle E_{\lambda}f,g\rangle.
\end{equation}

It is customary to write symbolically
\begin{equation}\label{Lsymbolically}
L=\int_{0}^{\infty}\lambda dE_{\lambda},
\end{equation}
the so-called \textit{spectral decomposition} of $L$.
\vspace{0.2cm}

The next logical step is the \textit{functional calculus} associated with the operator $L$; see also \cite[Theorem VIII.5]{RS}.

Let $g:\bR_{+}\rightarrow\bR$ a Borel measurable function. Then  the  operator $g(L)$ defined on
\begin{equation}
\text{Dom}(g):=\big\{\varphi\in\bL^2(\cM,d\mu):\;\int_{0}^{\infty}|g(\lambda)|^2 d\langle E_{\lambda}\varphi,\varphi\rangle<\infty\big\},
\end{equation}
as 
\begin{equation}
\langle g(L)\varphi,\psi\rangle =\int_{0}^{\infty}g(\lambda) d\langle E_{\lambda}\varphi,\psi\rangle,\quad\text{for every }\varphi\in\text{Dom}(g),\;\psi\in\bL^2(\cM,d\mu),
\end{equation}
is a self-adjoint operator mapping real-valued functions to real-valued functions. If $g$ is further assumed to be bounded, then $\text{Dom}(g)=\bL^2(\cM,d\mu)$ and $g(L):\bL^2(\cM,d\mu)\rightarrow\bL^2(\cM,d\mu)$ is a bounded operator. The operator $g(L)$ is referred as the \textit{spectral multiplier associated with} $g$ \textit{and} $L$ and it is symbolically expressed ---in the spirit of (\ref{Lsymbolically})--- as

\begin{equation}\label{g(L)symbolically}
g(L)=\int_{0}^{\infty}g(\lambda) dE_{\lambda}.
\end{equation}

The above spectral multipliers can take an explicit form for particular operators $L$ and metric spaces, as we will see in \S \ref{S:examples}.

For the purpose of the study of kernel density estimators we turn our attention to spectral multipliers associated with the operator $\sqrt{L}$, which is well-defined and self-adjoint; see \cite{Yosida}. The exact reasons behind the switch to $\sqrt{L}$ are discussed in \cite[Remark 2.2.$(\alpha)$]{CGP}. 

We denote by $\{F_{\lambda}:\lambda\ge0\}$ the spectral resolution of $\sqrt{L}$. Then $F_{\lambda}=E_{\lambda^2}$ and for every Borel measurable $g:\bR_{+}\rightarrow\bR$ it holds
\begin{equation}\label{g(sqrtL)}
g(\sqrt{L})=\int_{0}^{\infty}g(\lambda)dF_{\lambda}=\int_{0}^{\infty}g(\sqrt{\lambda})dE_{\lambda}.
\end{equation}

\subsection*{Summary and Notation}
For the rest of our study we fix the following terminology and notation.

As a \textit{symbol} we will refer to a Borel measurable and bounded function 
$$k:\bR_+\rightarrow\bR\quad\text{(symbol)}.$$ 

The spectral multiplier associated with the symbol $k$ and the operator $\sqrt{L}$ as in (\ref{g(sqrtL)}) will be denoted by the corresponding capital letter: $$K:=k(\sqrt L):\bL^2(\cM,d\mu)\rightarrow\bL^2(\cM,d\mu)\quad \text{(spectral multiplier)}.$$

By the above discussion, the operator $K$:

\noindent (i) is bounded on $\bL^2(\cM,d\mu)$, 

\noindent (ii) is self-adjoint, and 

\noindent (iii) maps real-valued functions to real-valued functions.
\vspace{0.2cm}

For the purpose of kernel density estimation we are interested in the following class of operators: We say that $K$ is an \textit{integral operator}, when there exists a measurable function $\cK(x, y)$ ---referred as the \textit{kernel} of the operator $K$--- 
$$\cM\times\cM\ni(x,y)\mapsto \cK(x,y)\in\bR\quad(\text{kernel})$$ such that
\begin{equation}
K(f)(x)=\int_{\cM}\cK(x,y)f(y)d\mu(y),\quad f\in\text{Dom}(K),\;x\in\text{Dom}(f).
\end{equation}
Note further that when the spectral multiplier $K=k(\sqrt{L})$ is an integral operator, then its kernel is real valued and symmetric; $$\cK(x,y)=\cK(y,x)\in\bR.$$
Such kernels are exactly the objects we will use for the kernel density estimation.
\vspace{0.2cm}

As always we need a notion of \textit{dilations} suitable for use in the current framework.

Let $k:\bR_{+}$ a symbol, $K$ the spectral multiplier associated with $k$ and $\sqrt{L}$ and assume that $K$ is an integral operator with kernel $\cK(x,y)$, as above. For every $h>0$ we denote by 
\vspace{0.2cm}

\noindent (i) $k_{h}(\lambda):=k(h\lambda),\;\lambda\in\bR_+$, the symbol induced by $k$ dilated by $h$. 

\noindent (ii) $K_h=k_h(\sqrt{L})=k(h\sqrt{L})$, the spectral multiplier associated with $k_h$ and $\sqrt{L}$. 

\noindent (iii) $\cK_h(x,y)=\cK_h(y,x)$, the symmetric real-valued kernel of $K_h$.

\vspace{0.2cm}

We shall need the following result from smooth functional calculus induced by the heat kernel,
developed in \cite{CKP,KP}. We fix the following notation first: Let $\ddelta>0$ and $\mm>0$. We denote by
\begin{equation}
\label{KernelD}
\cD_{\ddelta,\mm}(x,y):=\ddelta^{-d}\big(1+\ddelta^{-1}\rho(x,y)\big)^{-\mm},\quad\text{for}\;x,y\in \cM.
\end{equation}


\begin{theorem}\label{thm:S-local-kernels}
Suppose $k:\bR_{+}\rightarrow\bR$ is a symbol such that: $k\in \cC^\mm(\bR_+)$, $\mm>d$, 
\begin{equation}
\label{decay}
|k^{(\nu)}(\lambda)|\le C_\mm(1+\lambda)^{-r} \quad\text{for every}\; \lambda\ge 0\;\text{and}\;0\le \nu\le \mm,\;\text{where}\; r > \mm+d,
\end{equation}
and $k^{(2\nu+1)}(0)=0$ for every $\nu\ge 0$ such that $1\le 2\nu+1 \le \mm$.

Then $K_\ddelta$, $\ddelta>0$, is an integral operator with kernel $\cK_\ddelta (x, y)$
satisfying
\begin{equation}
\label{kernel-decay-D}
\big|\cK_\ddelta (x, y)\big|
\le cC_\mm \cD_{h,\mm}(x,y),
\end{equation}
where
$c>0$ is a constant depending on $\mm$ and
the structural geometric constants of the setting.

Moreover, for every $\ddelta>0$ and $x\in \cM$
\begin{equation}\label{Th2.2eq1}
\int_{\cM} \cK_\ddelta (x, y)d\mu(y)=k(0).
\end{equation}
\end{theorem}
\begin{remark}\label{Remark:FC} Let us comment on Theorem \ref{thm:S-local-kernels}.

$(\alpha)$ Let $k\in C^{\mm}(\bR)$ be an even function. Then the assumption $k^{(2\nu+1)}(0)=0$, $0<2\nu+1\le\mm$, holds automatically.

$(\beta)$ Let $k\in\cC^{\mm}(\bR_+)$ be such that $\supp k\subset[0,b]$, for some $b>0$. Then (\ref{decay}) holds for $C_{\mm}:=(1+b)^r\max\{\|k^{(\nu)}\|_{\infty}:0\le\nu\le\mm\}$. 

$(\gamma)$ Let us return to the setting's Assumption II and seed more light on it. The heat kernel $p_t(x,y)$ consists the kernel of the operator $e^{-tL}$. At this point, being more familiar with Spectral Theory, we define $k(\lambda):=e^{-\lambda^2}$. Then for every $t>0$ the operator $e^{-tL}$ is just the spectral multiplier $K_{\sqrt{t}}$, which is an integral operator by Theorem \ref{thm:S-local-kernels}, and the heat kernel equals $p_{t}(x,y)=\cK_{\sqrt{t}}(x,y)$.
\end{remark}

\subsection{Examples}\label{S:examples} We present the most basic examples of spaces $(\cM,\rho,\mu,L)$ falling under our umbrella. For more examples we refer to \cite{CGKPP,KPX} and the references therein. In the following spaces we also express the kernels obtained in an abstract sense of existence in Theorem \ref{thm:S-local-kernels}.

\begin{example}\label{exampleRn}
Let $\cM=\bR^d$ the Euclidean space associated with the operator $L=-\Delta$, the negative Laplacian. By default this space is included in our study. 

We proceed to the kernels. Let $k:\bR_{+}\rightarrow\bR$ a symbol satisfying the assumptions of Theorem \ref{thm:S-local-kernels}. We extend the symbol $k$ on $\bR^d$, radially; $\tilde{k}(\xi):=k(|\xi|)$, for every $\xi\in\bR^d$. Then the spectral multiplier $K=k(\sqrt L)$ is nothing but the Fourier multiplier associated with the symbol $\tilde{k}(\xi)$. Denote by $\hat{f}$ and by $\cF^{-1}f$ the Fourier transform and the inverse Fourier transform of the function $f:\bR^d\rightarrow\bR$, respectively. Then:
\begin{align*}
K(f)(x)&=\cF^{-1}(\tilde{k}\hat{f})(x)
\\
&=\kappa\ast f(x),\quad \text{where}\;\;\kappa:=\cF^{-1}\tilde{k}
\\
&=\int_{\bR^d}\kappa(x-y)f(y)dy,\quad \text{so}\;\;\cK(x,y)=\kappa(x-y).
\end{align*}

The kernel $\cK_{h}(x,y)$, by the properties of the Fourier transform, is the familiar
\begin{equation}
\label{kernelonRd}
\cK_{h}(x,y)=\frac{1}{h^d}\kappa\left(\frac{x-y}{h}\right),\quad x,y\in\bR^d,\;h>0.
\end{equation}

This example sheds light on the notion of spectral multipliers. Specifically, on $\bR^d$, they are the well-known Fourier multipliers, and the corresponding kernels $\cK(x,y)$ are the convolution kernels of the symbol $\kappa=\cF^{-1}\tilde{k}$. 
Moreover, the existing knowledge on $\bR^d$, together with the present correspondence, acts as a guide for the several developments on the setting of metric spaces associated with operators.
\end{example}

\vspace{0.2cm}

In the next examples we consider spaces $\cM$ of finite measure. In this case ---as it has been proved in \cite[Proposition 3.20]{CKP}--- the operator $L$ presents a discrete spectrum $0\le \lambda_0<\lambda_1<\cdots$. This implies the discrete decomposition
\begin{equation*}
\mathbb{L}^2=\bigoplus_{\nu=0}^{\infty} E_{\nu};\quad E_{\nu}:={\rm{ker}}(L-\lambda_{\nu}I),\;\nu\ge0.
\end{equation*} 

Let $\{e_{i}^{\nu}\}_{i=1,\dots,d_{\nu}}$ be an orthonormal basis of the eigenspace $E_{\nu}$ and $d_{\nu}:={\rm{dim}}E_{\nu}<\infty$, $\nu\ge0$. Then we have the projector operators
\begin{equation*}
P_{\nu}(x,y):=\sum_{i=1}^{d_{\nu}}e_i^{\nu}(x)\overline{e_i^{\nu}(y)},\quad x,y\in\cM,\;\nu\ge0.
\end{equation*}

Let $k:\bR_{+}\rightarrow\bR$ a symbol satisfying the assumptions of Theorem \ref{thm:S-local-kernels}. The corresponding spectral multiplier $K_h$, $h>0$, has the following kernel
\begin{equation}\label{kernelforgeneraldiscrete}
\cK_{h}(x,y)=\sum_{\nu=0}^{\infty}k(h\sqrt{\lambda_{\nu}})P_{\nu}(x,y),\quad x,y\in\cM.
\end{equation}

\noindent For more details we refer to \cite{CaKPi,KPX}. 

Importantly, when dealing with a specific metric measure space of finite measure, associated with an operator $L$, we just need to know (i) the eigenvalues and (ii) the projector operators, and then we get the kernels in (\ref{kernelforgeneraldiscrete}).

Next, we present precise expressions of (\ref{kernelforgeneraldiscrete}) on the cases of the unit sphere and the unit ball of $\bR^3$, which seems to be the most applicable. 
\vspace{0.2cm} 

\begin{example}\label{exampleSn}
Let $\cM=\bS^2$ the unit sphere of $\bR^{3}$ associated with the angular distance, the spherical measure and the spherical Laplacian. Then this space satisfies our Assumptions I and II; \cite{KPX}. The kernel takes the form:
\begin{equation}
\label{kernelons2}
\cK_h(\xi,\eta)=\sum_{\nu=0}^{\infty}\frac{1+2\nu}{4\pi}k\big(h\sqrt{\nu(\nu+1)}\big)P_{\nu}\big(\langle\xi,\eta\rangle\big),\quad\xi,\eta\in\bS^2,
\end{equation}
where $P_{\nu}$ the Legendre polynomials and $\langle\cdot,\cdot\rangle$ the inner product on $\bR^3$.

\end{example}

\begin{example}\label{exampleBn}
The unit ball $\bB^3$ of $\bR^3$ equipped with the distance \cite{DaiXu}
\begin{equation}\label{distanceB3}
\rho(x,y)=\arccos\big(\langle x,y\rangle+\sqrt{1-|x|^2}\sqrt{1-|y|^2}\big),
\end{equation}
the measure
\begin{equation}\label{measureB3}
d\mu(x)=\big(1-|x|^2\big)^{-1/2}dx,
\end{equation}
and the operator
\begin{align}\label{operatorB3}
L&=-\sum_{i=1}^{3}\partial_i^2
+\sum_{i,j=1}^{3}x_i x_j \partial_i \partial_j
+3\sum_{i=1}^{3}x_i \partial_i,
\end{align}
satisfies the assumptions of our setting; \cite{KPX}.

Expanding the discussion in \cite{CGKPP} and using \cite{KPY}, the kernel takes the form
\begin{equation}
\label{kernelonb3}
\cK_{h}(x,y)=\sum_{\nu=0}^{\infty}\frac{1+\nu}{2\pi^2}k\big(h\sqrt{\nu(\nu+2)}\big)G_{\nu}(x,y),\quad x,y\in\bB^3,
\end{equation}
where
\begin{align}\label{Gn}
G_{\nu}(x,y):=C_{\nu}^{1}&\big(\langle x,y\rangle+\sqrt{1-|x|^2}\sqrt{1-|y|^2}\big)
\nonumber
\\
&+C_{\nu}^{1}\big(\langle x,y\rangle-\sqrt{1-|x|^2}\sqrt{1-|y|^2}\big)
\end{align}
and $C_{\nu}^1$ the Gegenbauer polynomials of order $1$.
\end{example}

We emphasize that the kernels existing by Theorem \ref{thm:S-local-kernels} may look completely different as in (\ref{kernelonRd}), (\ref{kernelons2}), and (\ref{kernelonb3}); however, all of them enjoy the decay in (\ref{kernel-decay-D}), which is sharp in all the above cases as it can be confirmed by the properties of Fourier transform, Legendre polynomials and Gegenbauer polynomials, respectively. 

The advantage of the general theory is that it unifies spaces of different nature, extracts general results and expresses them in the particular cases of interest.

\subsection{Kernel density estimators on $\boldsymbol{(\cM,L)}$}

We are now ready to present kernel density estimators on the current general setting as introduced in \cite{CGKPP}.
\begin{definition}\label{D:kdeM} Let $n\in\mathbb{N}$ and $X_1,\dots,X_n$ be iid random variables on $\cM$. Let $k:\mathbb{R}_+\rightarrow \mathbb{R}$ be a symbol satisfying the assumptions of Theorem \ref{thm:S-local-kernels}, as well as $k(0)=1$, and $h>0$ a bandwidth. The associated kernel density estimator (kde) is defined as 
\begin{equation}\label{kde2}
\hat{f}_{n,h}(x):=\hat{f}_{n,h}(X_1,\dots,X_n;x):=\frac{1}{n}\sum\limits_{i=1}^{n} \cK_{h}(X_i,x), \quad x\in \cM.
\end{equation}
\end{definition}

Note that (\ref{kde2}) is well-defined for every $k$, as guaranteed by Theorem \ref{thm:S-local-kernels}. In addition (\ref{Th2.2eq1}) implies the fundamental property 
$$\int_{\cM}\cK_{h}(x,y)d\mu(y)=k(0)=1,$$ 
which is a standard assumption for the kernels in the Euclidean setting.

We further express explicitly the kde in (\ref{kde2}) on $\bR^d$, $\bS^2$ and $\bB^3$, just by expanding the Examples \ref{exampleRn}, \ref{exampleSn} and \ref{exampleBn}. Let $k$ a symbol as in Definition \ref{D:kdeM} and $h>0$. 
\vspace{0.2cm}

$(\alpha)$ When $\cM=\bR^d$, and $L=-\Delta$,
\begin{equation}\label{kdeRn}
\hat{f}_{n,h}(x)=\frac{1}{n}\frac{1}{h^d}\sum\limits_{i=1}^{n} \kappa\Big(\frac{X_i-x}{h}\Big), \quad x\in\bR^d,\;\;\text{where}\;\;\kappa=\cF^{-1}\tilde{k},
\end{equation}
which is the very well-known form of a kde on $\bR^d$.

$(\beta)$ When $\cM=\bS^2$, equipped with the angular distance, the spherical measure and the spherical Laplacian,
\begin{equation}
\label{kdes2}
\hat{f}_{n,h}(\xi)=\frac{1}{n}\sum\limits_{i=1}^{n}\sum_{\nu=0}^{\infty}\frac{1+2\nu}{4\pi}k\big(h\sqrt{\nu(\nu+1)}\big)P_{\nu}\big(\langle\xi,X_i\rangle\big),\quad\xi\in\bS^2.
\end{equation}

$(\gamma)$ When $\cM=\bB^3$, equipped with the distance in (\ref{distanceB3}), the measure in (\ref{measureB3}) and the operator in (\ref{operatorB3}),
\begin{equation}
\label{kdeb3}
\hat{f}_{n,h}(x)=\frac{1}{n}\sum\limits_{i=1}^{n}\sum_{\nu=0}^{\infty}\frac{1+\nu}{2\pi^2}k\big(h\sqrt{\nu(\nu+2)}\big)G_{\nu}(x,X_i),\quad x\in\bB^3,
\end{equation}
where $G_{\nu}$ as in (\ref{Gn}).

\subsection{H\"{o}lder spaces}\label{sec:holder}

We are closing this review by presenting some regularity spaces. In nonparametric estimation we assume that the density under study belongs to large regularity spaces. Regularity spaces on $\bR$ and $\bR^d$ have been studied for a century within many scientific disciplines. Historically, the first way to express the notion of regularity (or smoothness) was in terms of derivatives, and gradually Fourier transforms and convolutions extended such notions. For the historical path, we refer the reader to \cite{T}. 

H\"{o}lder spaces are a suitable choice for the purpose of pointwise density estimation (see \cite{Tsybakov}) that we will obtain in the present study. Let us recall this class on $\bR^1$: Let $s>0$ and denote by $\ell:=\lfloor s\rfloor$ the greatest integer strictly less than $s$. The H\"{o}lder space $\dot{\cH}^s(\bR)$ is the set of function $f:\bR\rightarrow\bR$ that are $\ell$-times differentiable and 
\begin{equation}
\label{HsR}
|f^{(\ell)}(x)-f^{(\ell)}(y)|\le c|x-y|^{s-\ell},
\end{equation}
for some constant $0\le c<\infty$ and every $x\neq y$. Note that slightly different versions of these spaces can be found in different sources, but the overall purpose is more or less the same.

\vspace{0.2cm}

We must define a suitable extension of (\ref{HsR}) on a metric space. For the right hand side, we simply use a power of the distance $\rho(x,y)$. Metric spaces lack the notion of derivatives, so a substitute for the left side is more challenging, but a solution comes from the operator $L$. In all of our examples in Section \ref{S:examples} we observe that the differentiability is linked with the definition of $L$. We also note that in every case presented in Section \ref{S:examples} $L$ is a differential operator of order 2. These facts justify the following definition:

\begin{definition}\label{D:HS}
Let $s>0$ and denote by $\ell=\lfloor s\rfloor$. The H\"{o}lder space of order $s$, $\dot{\cH}^s$, is defined as the set of all functions $f:\cM\rightarrow\bR$ such that
\begin{equation}\label{Ho1}
\|f\|_{\dot{\cH}^s}:=\sup\limits_{x\neq y} \frac{\big|L^{\ell/2}f(x)-L^{\ell/2}f(y)\big|}{\rho(x,y)^{s-\ell}}<\infty.
\end{equation}
\end{definition}

For the connection between these spaces and other smoothness spaces in our setting, we refer to \cite{CKP}. For the use of regularity spaces in Nonparametric Statistics in this generality, we further refer to \cite{CaKPi,CGKPP,CGP}.

\section{Pointwise density estimation}\label{Sec:MSE} We proceed to present some new results. Namely the pointwise estimation of densities enjoying H\"{o}lder regularity.

One of the main ways to measure the accuracy of the estimator $\hat{f}_{n,h}(x)$ at a given point $x\in\cM$ is by the \textit{mean squared error} (MSE):
\begin{equation}
\text{MSE}=\text{MSE}(\hat{f}_{n,h}(x)):=\mathbb{E}\big[\big(\hat{f}_{n,h}(x)-f(x)\big)^2\big],\quad x\in \cM,
\end{equation}
where $\mathbb{E}$ is the expectation of $(X_1,\dots,X_n),$ i.e.
\begin{eqnarray}\label{MSE1}
\text{MSE}&=&\mathbb{E}\big[\big(\hat{f}_{n,h}(x)-f(x)\big)^2\big]
\nonumber
\\
&=&\int_M\cdots\int_M \big(\hat{f}_{n,h}(x;x_1,\dots,x_n)-f(x)\big)^2 f(x_1)\cdots f(x_n)d\mu(x_1)\cdots d\mu(x_n).
\end{eqnarray}

What we are called to do is to determine the proper assumptions on the symbols $k$, so that the MSE of the corresponding kernel density estimator $\hat{f}_{n,h}$ to be optimally estimated, provided that the unknown density $f$ belongs to a certain H\"{o}lder space.

The main new result of this paper is the following:
\begin{theorem}\label{Th:MSEbound}
Let $s>0,$ $f\in L^{\infty}\cap\dot{\cH}^s$ and a symbol $k \in \cC^{\tau}(\mathbb{R}_+)$ for some $\tau>d+s$, satisfying: $k(0)=1$,
\begin{equation}
\label{moments}
k^{(\nu)}(0)=0,\quad\text{for every}\;\;1\le\nu\le\tau,
\end{equation}
and for some $r> \tau+d$,
\begin{equation}\label{asPV2}
|k^{(\nu)}(\lambda)|\leq C_{\tau}(1+\lambda)^{-r},\quad \text{for every}\;\;\lambda\ge0,\;0\leq \nu\leq \tau.
\end{equation}

We pick $h=h_n=n^{-\frac{1}{2s+d}}$. Then for every $n\in\bN$ the corresponding kde $\hat{f}_{n,h}$ satisfies
\begin{equation}\label{Th1}
\sup\limits_{x\in \cM} {\rm MSE}(\hat{f}_{n,h}(x))=\sup\limits_{x\in \cM} \mathbb{E}\big[\big(\hat{f}_{n,h}(x)-f(x)\big)^2\big]\leq c C(f) n^{-\frac{2s}{2s+d}},
\end{equation}
where the constant $c>0$, depends only on $\tau,\;s,\;C_{\tau}$ and the structural constants of the setting, while $C(f)$ is given by
\begin{equation}\label{C(f)}
C(f):=\max\big(\|f\|_{\infty},\|f\|_{\dot{\cH}^s}^2\big).
\end{equation}
\end{theorem}

\vspace{0.2cm}

While approaching the proof of Theorem \ref{Th:MSEbound} we will have the opportunity to present the action on the setting and highlighting the correspondence with the classical Euclidean framework.

We first take a closer look at the assumptions on the symbol generating the kernels. We restrict ourselves to the example of $\cM=\bR^d$. As we saw in Example \ref{exampleRn}, the radial extension $\tilde{k}$ of the symbol $k$ is the Fourier transform of the function $\kappa$, which yields the usual kde as in (\ref{kdeRn}). Translating the assumptions of Theorem \ref{Th:MSEbound} in the Fourier transform language we get the usual assumptions on the $\kappa$ for Euclidean spaces;

$(\alpha)$ Assumption (\ref{moments}), means simply that the $\kappa$ enjoys vanishing moments up to some certain order.

$(\beta)$ Assumption (\ref{asPV2}), thanks to Theorem \ref{thm:S-local-kernels} and (\ref{kernelonRd}), ensures that 
\begin{align}
\int_{\bR^d}(1+|\xi|)^s|\kappa(\xi)|d\xi&=\int_{\bR^d}(1+|\xi|)^s|\cK(\xi,0)|d\xi
\le c\int_{\bR^d}(1+|\xi|)^s\cD_{1,\tau}(\xi,0)d\xi
\nonumber
\\
&=c\int_{\bR^d}(1+|\xi|)^{-(d+\varepsilon)}d\xi,\quad \varepsilon:=\tau-d-s>0
\nonumber
\\
&=c_d\int_{0}^{\infty}\frac{\varrho^{d-1}d\varrho}{(1+\varrho)^{d+\varepsilon}}\quad\text{(polar coordinates)}
\nonumber
\\
&\le c\int_{0}^{\infty}\frac{d\varrho}{(1+\varrho)^{1+\varepsilon}}<\infty.
\label{finiteorder}
\end{align}

$(\gamma)$ The assumption $k(0)=1$, simply asserts that
$$\int_{\bR^d}\kappa(\xi)d\xi=\hat{\kappa}(0)=\tilde{k}(0)=k(0)=1.$$
\vspace{0.2cm}

The standard approach for dealing with the MSE is to decompose it as follows:
\begin{equation}\label{MSE2}
\text{MSE}(\hat{f}_{n,h}(x))=\sigma^2 (x)+b^2 (x)
\end{equation}
where the function $\sigma^2 (x)$ is the \textit{variance} of the estimator $\hat{f}_{n,h}(x)$, i.e.
\begin{equation}\label{V1}
\sigma^2(x):=\mathbb{E}\big[\big(\hat{f}_{n,h}(x)-\mathbb{E}[\hat{f}_{n,h}(x)]\big)^2\big],\quad x\in \cM
\end{equation}
and $b(x)$ is the \textit{bias} of $\hat{f}_{n,h}(x),$ i.e.
\begin{equation}\label{B1}
b(x):=\mathbb{E}\big[\hat{f}_{n,h}(x)\big]-f(x), \quad x\in \cM.
\end{equation}

We will separate the proof of Theorem \ref{Th:MSEbound} in the two usual steps: the estimation of the variance and the estimation of the bias, but before the proof, we provide some remarks below.

\begin{remark} 
$(\alpha)$ Another form for the conclusion of Theorem \ref{Th:MSEbound} is
\begin{equation}\label{Th11}
\sup_{f\in \mathbb{F}^{s}(m)}\sup\limits_{x\in M} \mathbb{E}\Big[\big(\hat{f}_{n,h}(x)-f(x)\big)^2\Big]\leq cn^{-\frac{2s}{2s+d}},
\end{equation}
where $\mathbb{F}^{s}(m):=\{f\in L^{\infty}\cap \dot{\cH}^{s}: \|f\|_{\infty}\le m\;\text{and}\;\|f\|_{\dot{\cH}^{s}}\le m\}$, $m>0$, and the constant $c>0$ depends also on $m>0$.

$(\beta)$ For latter use we state that our choice of $h_n$ gives $h_n\rightarrow0$ and $nh_n^d\rightarrow \infty$ when $n\rightarrow \infty$.

$(\gamma)$ If the symbol $k$ is compactly supported, then it obviously satisfies (\ref{asPV2}); see Remark \ref{Remark:FC} $(\beta)$.

$(\delta)$ The rate obtained in Theorem \ref{Th:MSEbound} is the optimal one (see e.g. \cite{Tsybakov}).
\end{remark}

The following simple inequality is established in \cite{CKP} under more general assumptions. Here we express it for Ahlfors regular spaces and we give its proof for having the opportunity to present some first calculations on metric spaces:
\begin{lemma}\label{L:tech}If $\tau >d$, there exists a constant $c=c_\tau>0$ such that for every $\ddelta>0$ and $x\in \cM$
\begin{equation}\label{tech-1}
I_{h,\tau}(x):=\int_{\cM} \big(1+\ddelta^{-1}\rho(x, y)\big)^{-\tau} d\mu(y) \le ch^{d}.
\end{equation}
\end{lemma}
\begin{proof}
We split the metric space as $$\cM=\bigcup_{\nu=0}^{\infty}M_{\nu},$$
where $M_0:=B(x,h)$ and $M_{\nu}:=B(x,2^{\nu}h)\setminus B(x,2^{\nu-1}h)$, for every $\nu\in\bN$.

Then 
$$I_{h,\tau}(x)=\sum_{\nu=0}^{\infty}\int_{M_{\nu}}\big(1+\ddelta^{-1}\rho(x, y)\big)^{-\tau} d\mu(y).$$

Of course,
$$\int_{M_{0}}\big(1+\ddelta^{-1}\rho(x, y)\big)^{-\tau}\le|B(x,h)|\le c_1 h^{d},$$
thanks to (\ref{Ahlfors}).

Let $\nu\in\bN$ and $y\in M_{\nu}\subset B(x,2^{\nu-1}h)^c$. Then $1+\ddelta^{-1}\rho(x, y)\ge 1+2^{\nu-1}> 2^{\nu-1}$. This together with (\ref{Ahlfors}) implies that
\begin{align*}\int_{M_{\nu}}\big(1+\ddelta^{-1}\rho(x, y)\big)^{-\tau} d\mu(y)&\le 2^{-(\nu-1)\tau}|M_{\nu}|\le 2^{\tau}2^{-\nu\tau}|B(x,2^{\nu}h)|
\\
&\le c_1 2^{\tau}2^{-\nu(\tau-d)}h^{d}.
\end{align*}

Combining all the above and since we assumed that $\tau>d$ we conclude to
\begin{align*}
I_{h,\tau}(x)&\le c_1 2^{\tau}\sum_{\nu=0}^{\infty}2^{-\nu(\tau-d)}h^{d}=c_1 2^{\tau}\frac{1}{1-2^{-\tau+d}}h^{d}
\\
&=c_1\frac{2^{2\tau}}{2^{\tau}-2^{d}}h^{d}=:c_{\tau}h^{d}.
\end{align*}
\end{proof}
\vspace{0.2cm}

Let us point out that the above estimate is sharp in the sense that
\begin{equation}
I_{h,\tau}(x)\ge \int_{M_{0}}\big(1+\ddelta^{-1}\rho(x, y)\big)^{-\tau} d\mu(y)\ge \frac{2^{-\tau}}{c_1}h^{d},
\end{equation}
thanks to (\ref{Ahlfors}). 

It is well known that such an integral is classical on the Euclidean space and can be handled using (generalized) polar coordinates, exactly as we did in (\ref{finiteorder}). On an abstract Ahlfors regular metric space it can be sharply estimated as above.  

\subsection{Estimation of the variance}
We proceed to the first step of the proof of Theorem \ref{Th:MSEbound}. We will estimate the variance of bounded densities. As always there is not any regularity required. The main tools are Theorem \ref{thm:S-local-kernels} and Lemma \ref{L:tech}.

\begin{proposition}\label{P:V}
Let $f\in L^{\infty}$, $\tau>d$ and a multiplier $k\in \cC^{\tau}(\mathbb{R}_+)$ satisfying (\ref{moments}) and (\ref{asPV2}). Then for every $x\in \cM$, $0<h\leq 1$ and $n\in\mathbb{N}$
\begin{equation}\label{EstV}
\sigma^2(x)\leq \frac{C_1}{nh^d}\|f\|_{\infty},
\end{equation}
where the constant $C_1>0$ depends only on $\tau,\;C_\tau$ and the structural constants.
\end{proposition}

\begin{proof} Recalling the results expanded in \S \ref{kdemetric}, the spectral multiplier $K_h$ associated with the dilated symbol $k_h(\lambda)=k(h\lambda)$, is an integral operator with kernel $\cK_h(x,y)$, $x,y\in\cM$.

We introduce the random variables
\begin{equation}\label{etai}
\eta_i(x):=\cK_h(X_i,x)-\mathbb{E}\big[\cK_h(X_i,x)\big],\quad x\in \cM,\;i=1,\dots,n
\end{equation}
and we observe that $\eta_1(x),\dots,\eta_n(x)$ are iid random variables with $\mathbb{E}[\eta_i(x)]=0,$ for $i=1,\dots,n$. For their variance we have
\begin{eqnarray}\label{varfirst}
\mathbb{E}[\eta_i^2(x)]&=& \mathbb{E}\big[\big(\cK_h(X_i,x)\big)^2\big]-
\left(\mathbb{E}\big[\cK_h(X_i,x)\big]\right)^2\\ \nonumber
&\leq& \mathbb{E}\big[\big(\cK_h(X_i,x)\big)^2\big] \\ \nonumber
&=&\int_\cM |\cK_h(x,y)|^2 f(y) d\mu(y).
\end{eqnarray}

By Theorem \ref{thm:S-local-kernels} we have the certain bounds
\begin{equation}
\label{KernelvariancewithAR}
|\cK_h(x,y)|\le cC_{\tau}\cD_{h,\tau}(x,y).
\end{equation}

Since $f\in L^{\infty}$ we derive
\begin{eqnarray}\label{ppV1}
\mathbb{E}[\eta_i^2(x)]&\leq& c \|f\|_{\infty} h^{-2d}\int_\cM\big(1+h^{-1}\rho(x,y)\big)^{-2\tau}d\mu(y)\\ \nonumber
&=&c \|f\|_{\infty} h^{-2d}I_{h,2\tau}\leq C_1 \|f\|_{\infty} h^{-d},
\end{eqnarray}
where for the ultimate inequality we used (\ref{tech-1}). 

We observe that 
\begin{equation}\label{ppV2}
\sum\limits_{i=1}^n \eta_i (x)= n\big(\hat{f}_{n,h}(x)-\mathbb{E}\big[\hat{f}_{n,h}(x)\big]\big).
\end{equation}
Bearing in mind that the independent random variables $\eta_i$ have zero mean and guided by (\ref{ppV1}) and (\ref{ppV2}) we arrive at
\begin{eqnarray}\label{varlast}
\sigma^2(x)&=&\mathbb{E}\big[\big(\hat{f}_{n,h}(x)-\mathbb{E}[\hat{f}_{n,h}(x)]\big)^2\big]
=\mathbb{E}\Big[\Big(\frac{1}{n} \sum\limits_{i=1}^n \eta_i (x)\Big)^2\Big]\\ \nonumber
&=&\frac{1}{n^2} \sum\limits_{i=1}^n\mathbb{E}\big[\eta_i^2(x)\big]
\leq C_1 \|f\|_{\infty}\frac{1}{nh^d}.
\end{eqnarray}
\end{proof}

\subsection{Estimation of the bias}
We will estimate the bias under the assumption that the pdf $f$ lies in the H\"{o}lder space.
\begin{proposition}\label{P:B}
Let $s>0,\;f\in L^{\infty}\cap\dot{\cH}^s$ and a multiplier $k\in\cC^{\tau}(\mathbb{R}_+)$, for some$\;\tau>d+s$, satisfying $k(0)=1$, (\ref{moments}) and (\ref{asPV2}). Then for every $x\in \cM$, $0<h\leq 1$ and $n\in\mathbb{N}$,
\begin{equation}\label{EstB}
|b(x)|\leq C_2 \|f\|_{\dot{\cH}^s} h^s,
\end{equation}
where the constant $C_2>0$ depends only on $s,\;\tau,\;C_{\tau}$ and the structural constants of the setting.
\end{proposition}

\begin{proof}
Since $X_i$ are iid with common density $f$, we obtain
\begin{eqnarray}\label{bias1}
b(x)&=&\mathbb{E}\big[\hat{f}_{n,h}(x)\big]-f(x)
\nonumber
\\
&=&\frac{1}{n}\sum\limits_{i=1}^n \mathbb{E}\big[\cK_h(X_i,x)\big]-f(x)
\\ \nonumber
&=&\big(K_h-I\big)f(x),
\end{eqnarray}
where $I$ the identity operator on $\cM$ and $K_h=k_h(\sqrt{L})$ the spectral multiplier associated with the dilated symbol $k_h$ and the operator $\sqrt{L}$, as in \S \ref{kdemetric}.

For the given bandwidth $0<h\le1$ there exists a unique integer $i\in\bN_0$ such that
\begin{equation}\label{i}
2^{-i}\le h<2^{-i+1}.
\end{equation}
 
We consider the symbol $\psi\in\cC^{\infty}(\bR_+)$ with $\supp\psi\subset[0,2]$, $\psi(\lambda)=1$, for every $\lambda\in[0,1]$ and $0\le \psi(\lambda)\le1$, for every $\lambda\in[0,2]$. 

We set $\varphi(\lambda):=\psi(\lambda)-\psi(2\lambda)$ which is $\cC^{\infty}$ and supported in $[2^{-1},2]$. 

By the construction of the above functions, it turns out that
$$\psi(2^{-i}\lambda)+\sum\limits_{j=i+1}^{\infty}\varphi(2^{-j}\lambda)=1,\;\;\text{for every}\;\lambda\in\bR_+.$$
Then by \cite[Corollary 3.9]{CKP}
\begin{equation}
\label{fsplit}
f=\Psi_{2^{-i}}f+\sum\limits_{j=i+1}^{\infty}\Phi_{2^{-j}}f,
\end{equation}
where by the capital $\Psi_{2^{-i}}$ and $\Phi_{2^{-j}}$ we denoted the spectral multipliers as in Section \ref{subsec:func-calc}; $\Psi_{2^{-i}}=\psi(2^{-i}\sqrt{L})$ and $\Phi_{2^{-j}}=\varphi(2^{-j}\sqrt{L})$.

We set $\ell:=\lfloor s\rfloor$ and we introduce the symbols
\begin{equation}
\label{w0}
g^i(\lambda):=\frac{(k(h2^{i}\lambda)-1)\psi(\lambda)}{|\lambda|^{\ell}}\;\;\text{and}\;\;g^j(\lambda):=\frac{(k(h2^{j}\lambda)-1)\varphi(\lambda)}{|\lambda|^{\ell}},\;j>i.
\end{equation}

We proceed to justify that the assumptions of Theorem \ref{thm:S-local-kernels} are fulfilled for the symbols $g^j$, $j\ge i$, using Remark \ref{Remark:FC} $(\beta)$.
\vspace{0.2cm}

By the fact that $k\in\cC^{\tau}(\bR_+)$, the values of the derivatives $k^{(\nu)}(0)$, $0\le\nu\le\tau$, (\ref{i}) and the definitions of the symbols $\psi$ and $\varphi$ we have that:

$g^j\in\cC^{\tau-\ell}(\bR_+)$, for every $j\ge i$. Of course $\tau-\ell>d+s-\ell>d$.

$\supp g^i\subset[0,2]$ and $\supp g^j\subset[2^{-1},2]$, for every $j>i$.

$g^i(0)=\lim_{\lambda\rightarrow0^{+}}\frac{k(h2^{i}\lambda)-1}{\lambda^{\ell}}\psi(\lambda)=\frac{(h2^{i})^{\ell}k^{(\ell)}(0)}{\ell!}\psi(0)=0$.

$(g^i)^{(2\nu+1)}(0)=0$, for every $1\le 2\nu+1\le \tau-\ell$.

Moreover by the vanishing derivatives' assumption (\ref{moments}), the decay in (\ref{asPV2}) and the bottom of the support of $\varphi$, we obtain after some calculus that
$$|(g^j)^{(\nu)}(\lambda)|\le c(\tau,\ell),\quad\text{for every}\;\lambda\ge0,\;0\le\nu\le\tau-\ell,\;j\ge i,$$
where the above constant $c(\tau,\ell)>0$ is independent of $j$.
\vspace{0.2cm}

\noindent By Theorem \ref{thm:S-local-kernels}, coupled with Remark \ref{Remark:FC}$(\beta)$, the spectral multipliers $G^j_{2^{-j}}=g^{j}(2^{-j}\sqrt L)$, $j\ge i$, are integral operators and their corresponding kernels $\cG^j_{2^{-j}}(x,y)$ present the behaviour
\begin{equation}
\label{Omegakernels}
|\cG^j_{2^{-j}}(x,y)|\le c\cD_{2^{-j},\tau-\ell}(x,y),\quad x,y\in\cM,\;j\ge i.
\end{equation}

By the definition of the symbols $g^j$, $j\ge i$ in (\ref{w0}) we get
\begin{equation}\label{w1}
(K_h-I)\Psi_{2^{-i}}f=2^{-\ell i}G^i_{2^{-i}}L^{\ell/2}f
\end{equation}
and
\begin{equation}\label{w2}
(K_h-I)\Phi_{2^{-j}}f=2^{-\ell j}G^j_{2^{-j}}L^{\ell/2}f,\;\;j>i.
\end{equation}

Combining (\ref{bias1}), (\ref{fsplit}) with (\ref{w1}) and (\ref{w2}) we obtain the expansion
\begin{equation}\label{bias11}
b(x)=\sum_{j=i}^{\infty}2^{-\ell j}G^j_{2^{-j}}L^{\ell/2}f(x).
\end{equation}

Since $G^j_{2^{-j}}$ are integral operators and because of $g^j(0)=0$, for every $j\ge i$, using (\ref{Th2.2eq1}) we express $G^j_{2^{-j}}L^{\ell/2}f(x)$ as
\begin{align}\label{w3}
G^j_{2^{-j}}L^{\ell/2}f(x)&=\int_{\cM}\cG^j_{2^{-j}}(x,y)L^{\ell/2}f(y)d\mu(y)
\nonumber
\\&=\int_{\cM}\cG^j_{2^{-j}}(x,y)\big(L^{\ell/2}f(y)-L^{\ell/2}f(x)\big)d\mu(y).
\end{align}

The membership of $f$ in the H\"{o}lder space $\dot{\cH}^s$ implies that
\begin{align}\label{Hproof}
\big|L^{\ell/2}f(y)-L^{\ell/2}f(x)\big|&\le \|f\|_{\dot{\cH}^s}\rho(x,y)^{s-\ell}
\nonumber
\\
&\le 2^{\ell j}2^{-sj}\|f\|_{\dot{\cH}^s}\big(1+2^{j}\rho(x,y)\big)^{s-\ell},\quad j\ge i.
\end{align}

We equip (\ref{bias11}) with (\ref{w3}), (\ref{Omegakernels}) and (\ref{Hproof}) to arrive at the expression
\begin{align}
\label{bias12}
|b(x)|&\le\sum_{j=i}^{\infty}2^{-\ell j}\int_{\cM}|\cG^j_{2^{-j}}(x,y)|\big|L^{\ell/2}f(y)-L^{\ell/2}f(x)\big|d\mu(y)
\nonumber
\\
&\le c\|f\|_{\dot{\cH}^s}\sum_{j=i}^{\infty}2^{-sj}\int_{\cM}\cD_{2^{-j},\tau-s}(x,y)d\mu(y)
\nonumber
\\
&= c\|f\|_{\dot{\cH}^s}\sum_{j=i}^{\infty}2^{-sj}2^{jd}I_{2^{-j},\tau-s}(x),
\end{align}
where $I_{2^{-j},\tau-s}(x)$, as in Lemma \ref{L:tech}. Thanks to the assumption $\tau>d+s$, by (\ref{tech-1}), the fact that $s>0$ and (\ref{i}) we conclude that

\begin{align}
\label{biasfinal}
|b(x)|&\le c\|f\|_{\dot{\cH}^s}\sum_{j=i}^{\infty}2^{-sj}2^{jd}I_{2^{-j},\tau-s}(x)
\nonumber
\\
&\le c\|f\|_{\dot{\cH}^s}\sum_{j=i}^{\infty}2^{-sj}\le c\|f\|_{\dot{\cH}^s}2^{-is}\le C_2\|f\|_{\dot{\cH}^s}h^{s}
\end{align}
and the proof is complete.
\end{proof}

\textit{End of the proof of Theorem} \ref{Th:MSEbound}. We combine Propositions \ref{P:V} and \ref{P:B} to conclude the proof of Theorem \ref{Th:MSEbound} in the standard way.
\subsection{Kernel density estimators on the sphere} The shape of the Earth justifies the unit sphere $\bS^2$ of $\bR^3$ as the most important domain for the purposes of several sciences. In the present paper we study earthquakes that are the subject of seismology, but many other sciences like astrophysics, environment and geology could be interested in this geometry too. We describe how the kernels obtained in Section \ref{S:examples} should be used in a data analysis.

We consider the symbols
\begin{equation}
\label{gsymbols}
g^{\sigma}(\lambda):=(1+|\lambda|^{\sigma})^{-1},\;\lambda\in\bR,
\end{equation}
for $\sigma\in\bN$, with $\sigma>1$. Evidently for every $\sigma>1$, the symbol $g^{\sigma}$ is an even function such that $g^{\sigma}\in\cC^{\sigma-1}(\bR)$, $g^{\sigma}(0)=1$, $(g^{\sigma})^{(\nu)}(0)=0$, for every $1\le\nu\le\sigma-1$ and presents the decay as in (\ref{asPV2}) for $r=\sigma$. Such symbols are suitable for generating kdes by choosing the appropriate value of $\sigma$ depending on the dimension $d$ and the regularity $s$ and then the appropriate bandwidth $h$ depending also on the datasize $n$.







For the purpose of our present study, we restrict our attention on the case of the unit sphere $\cM=\bS^2$.

Let $s>0$ and denote by $\lceil s\rceil$ the smallest integer strictly grater than $s$. The symbols (\ref{gsymbols}) for $r=\sigma:=5+\lceil s\rceil$ satisfy the assumptions of Theorem \ref{Th:MSEbound} for densities on $\dot{\cH}^s$.

The expression (\ref{kernelons2}) could be used by R (or Python etc) after the infinite series be truncated until some certain integer $N\in\bN$, namely
\begin{equation}
\label{kdes2trun}
\hat{f}_{n,h,N}(\xi):=\frac{1}{n}\sum\limits_{i=1}^{n}\sum_{\nu=0}^{N}\frac{1+2\nu}{4\pi}g^r\big(h\sqrt{\nu(\nu+1)}\big)P_{\nu}\big(\langle\xi,X_i\rangle\big),\quad\xi\in\bS^2.
\end{equation}

Note that
$$g^{r}(h\sqrt{\nu(\nu+1)})< (h\sqrt{\nu(\nu+1)})^{-r}<h^{-r}\nu^{-r},\quad\nu\in\bN.$$
Moreover, for the Legendre polynomials it is well-known that $|P_{\nu}(u)|\le 1,$ for every $u\in[-1,1]$ and of course
$\frac{1+2\nu}{4\pi}\le 0.51\frac{\nu}{\pi}$, for every $\nu\ge 25$.

Then the error (absolute value of the difference) because of the truncation of (\ref{kernelons2}) until the order $N\ge24$ can be safely bounded from above by
\begin{align}\label{eq:trunc_error_sphere}
\text{error}&\le\sum_{\nu>N}\frac{0.51 \nu}{\pi}h^{-r}\nu^{-r}=\frac{0.51}{\pi}h^{-r}\sum_{\nu=N+1}^{\infty}\nu^{-r+1}
\nonumber
\\
&\le \frac{0.51}{\pi}h^{-r}\int_{N}^{\infty}x^{-r+1}dx=\frac{0.51 h^{-r}N^{-r+2}}{\pi(r-2)}.
\end{align}

Recall that by Theorem \ref{Th:MSEbound} $h=n^{-1/(2s+2)}$, where $n$ is our datasize. 

Expression (\ref{eq:trunc_error_sphere}) asserts that an effectively large $N$ could provide a certain error-bound when the datasize $n$ and the regularity $s$ are considered as fixed.


As an example, take $s\in(0,1]$ (the less restrictive range) which corresponds to the value $r=6$. In this case and after setting $h=n^{-1/2(s+1)}$ the error is at most
\begin{equation}\label{Nfinal}
\frac{0.51*n^{3/(s+1)}}{4\pi}N^{-4}.
\end{equation}

In a specific data analysis, with a given datasize $n$, one should bear in mind to respect (\ref{Nfinal}) for the hypothetical smoothness' level $s$ and obtain appropriate values for the error is pre-defined as ``suitable". For a data analysis of earthquakes the reader is referred to Section \ref{S:data}.

\section{Spaces of homogeneous type}\label{Sec:Doubling}
We proceed to more general assumptions than those in Section \ref{Introduction}. Specifically, we no longer assume that our space enjoys the Ahlfors regularity, rather than the so-called doubling volume property (\ref{doubling-0}) below. Such a setting is what is referred as a \textit{space of homogeneous type}. 

We replace Assumption I with the following:

(a) {\em Doubling volume condition:} There exists a constant $c_0>1$ such that
\begin{equation}\label{doubling-0}
0 < |B(x,2r)| \le c_0|B(x,r)|<\infty
\quad\hbox{for all $x \in \cM$ and $r>0$,}
\end{equation}
where $|B(x,r)|$ is the volume of the open ball $B(x,r)$ centred at $x$ of radius $r$.

\smallskip

\noindent
(b) {\em Noncollapsing condition:}
There exists a~constant $c_1>0$ such that
\begin{equation}\label{non-collapsing}
\inf_{x\in \cM}|B(x,1)|\ge c_1.
\end{equation}

\vspace{0.2cm}

We modify Assumption II by replacing the factor $t^{-d/2}$ by
\begin{equation}
\big(|B(x,\sqrt{t})||B(y,\sqrt{t})|\big)^{-1/2}
\end{equation}
in equations (\ref{Gauss-local}) and (\ref{lip}).
\vspace{0.2cm}

Some remarks are in order:

$(\alpha)$ Of course (a) and (b) hold trivially true under (i).


$(\beta)$ From (\ref{doubling-0}) it follows that there exist $c_0'>0$ and $d>0$ such that
\begin{equation}\label{doubling}
|B(x,\lambda r)| \le c_0'\lambda^{d} |B(x,r)|
\quad\hbox{for every $x \in \cM$, $r>0$, and $\lambda >1$,}
\end{equation}
the constant $d$ above $d'$ is referred as the \textit{homogeneous dimension} of $(\cM,\rho,\mu)$. This generalizes effectively the Ahlfors dimension used in the previous sections.

$(\gamma)$ A connection between the volume of balls of small radius, with their radius and the dimension comes from (\ref{non-collapsing}) and (\ref{doubling}):
\begin{equation}\label{DNC}
|B(x, r)|\geq c r^d, \quad x\in \cM,\;0<r\leq 1.
\end{equation}

$(\delta)$ In the framework of spaces of homogeneous type we replace the kernels defined in (\ref{KernelD}) by
\begin{equation}
\label{KernelDd}
\cD_{\ddelta,\mm}(x,y):=\frac{\big(1+\ddelta^{-1}\rho(x,y)\big)^{-\mm}}{(|B(x,\ddelta)||B(y,\ddelta)|)^{1/2}},\quad\text{for}\;x,y\in \cM.
\end{equation}

On the background results: 

\noindent (i) Theorem \ref{thm:S-local-kernels} holds as it is, but with the kernels $\cD_{\ddelta,\mm}$ as in (\ref{KernelDd}). 

\noindent (ii) Lemma \ref{L:tech} takes the form: for every $\tau>d$, there exists a constant $c=c_{\tau}>0$ such that
\begin{equation}
\label{tech-2}
I_{h,\tau}\le c|B(x,h)|,\quad\text{for every}\;x\in\cM,\;h>0.
\end{equation}

$(\varepsilon)$ To compare the volumes of balls with different centers $x, y\in \cM$ and the same radius $r$, we note first that
$B(x,r) \subset B\big(y, \rho(y,x) +r\big)$, which coupled with (\ref{doubling}) leads to
\begin{equation}\label{D2}
|B(x, r)| \le c\big(1+ \rho(x,y)/r\big)^d  |B(y, r)|,
\quad x, y\in \cM, \; r>0.
\end{equation}

The last implies that the kernel in (\ref{KernelDd}) is estimated by:
\begin{equation}
\label{changecentersD}
D_{h,\tau}(x,y)\le c|B(x,h)|^{-1}(1+h^{-1}\rho(x,y))^{-\tau+d/2}.
\end{equation}
\vspace{0.2cm}

We are now in place to express the boundedness of the Mean Squared Error on the more general setting of spaces of homogeneous type associated with operators:
\begin{theorem}\label{Th:MSEbound2}
Let $s>0,$ $f\in L^{\infty}\cap\dot{\cH}^s$ and a symbol $k \in \cC^{\tau}(\mathbb{R}_+)$ for some $\tau>3d/2+s$, satisfying: $k(0)=1$,
\begin{equation}
\label{momentsd}
k^{(\nu)}(0)=0,\quad\text{for every}\;\;1\le\nu\le\tau,
\end{equation}
and for some $r> \tau+d$,
\begin{equation}\label{asPV2d}
|k^{(\nu)}(\lambda)|\leq C_{\tau}(1+\lambda)^{-r},\quad \text{for every}\;\;\lambda\ge0,\;0\leq \nu\leq \tau.
\end{equation}
We pick $h=h_n=n^{-\frac{1}{2s+d}}$. Then for every $n\in\bN$ the corresponding kde $\hat{f}_{n,h}$ satisfies
\begin{equation}\label{Th1d}
\sup\limits_{x\in \cM} \mathbb{E}\big[\big(\hat{f}_{n,h}(x)-f(x)\big)^2\big]\leq c C(f) n^{-\frac{2s}{2s+d}},
\end{equation}
where the constant $c>0$, depends only on $\tau,\;s,\;c_{\tau},\beta$ and the structural constants of the setting, while $C(f)$ is given by
\begin{equation}\label{C(f)d}
C(f):=\max\big(\|f\|_{\infty},\|f\|_{\dot{\cH}^s}^2\big).
\end{equation}
\end{theorem}

\begin{proof}
In the light of (\ref{MSE2}), we need to bound the variance and the bias in a similar manner as in Propositions \ref{P:V} and \ref{P:B}.

We start with the variance. 

By (\ref{KernelvariancewithAR}) and (\ref{changecentersD}) we get the behaviour
\begin{equation*}
|K_h(x,y)|\le c|B(x,h)|^{-1}\big(1+h^{-1}\rho(x,y)\big)^{-\tau+d/2},\quad\text{for every}\;x,y\in\cM.
\end{equation*}
The last replaced in (\ref{ppV1}) implies since $\tau>3d/2+s>d$,
\begin{eqnarray*}
\mathbb{E}[\eta_i^2(x)]&\leq& c \|f\|_{\infty} |B(x,h)|^{-2}\int_M\big(1+h^{-1}\rho(x,y)\big)^{-2\tau+d}d\mu(y)\\ \nonumber
&=&c \|f\|_{\infty} |B(x,h)|^{-2}I_{h,2\tau-d}\leq  c \|f\|_{\infty} |B(x,h)|^{-1} 
\le   C_1 \|f\|_{\infty} h^{-d},
\end{eqnarray*}
where we used (\ref{tech-2}) and (\ref{DNC}) respectively. Now the estimation of the variance in (\ref{EstV}) follows as in (\ref{ppV2}).
\vspace{0.2cm}

We proceed to bound the bias as in Proposition \ref{P:B}. Recall that $\ell:=\lfloor s\rfloor$. This time the kernels $\cG^j_{2^{-j}}(x,y)$ enjoy the behaviour
$$|\cG^j_{2^{-j}}(x,y)|\le c\cD_{2^{-j},\tau-\ell}(x,y)\le c|B(x,2^{-j})|^{-1}\big(1+2^{j}\rho(x,y)\big)^{-\tau+\ell+d/2},$$
thanks to (\ref{changecentersD}). Then
\begin{align}
\label{bias12d}
|b(x)|&\le\sum_{j=i}^{\infty}2^{-\ell j}\int_{\cM}|\cG^j_{2^{-j}}(x,y)|\big|L^{\ell/2}f(y)-L^{\ell/2}f(x)\big|d\mu(y)
\nonumber
\\
&\le c\|f\|_{\dot{\cH}^s}\sum_{j=i}^{\infty}2^{-sj}|B(x,2^{-j})|^{-1}I_{2^{-j},\tau-\frac{d}{2}-s}(x),
\nonumber
\\
&\le c\|f\|_{\dot{\cH}^s}\sum_{j=i}^{\infty}2^{-sj}\le C_2 h^{s},
\end{align}
where we used (\ref{tech-2}) which is valid because $\tau>3d/2+s$. 

The reminder of the proof is the same as in the proof of Theorem \ref{Th:MSEbound}.
\end{proof}

\vspace{0.2cm}

Let us close this section with some comments on the geometric assumptions.

Obviously the doubling property (\ref{doubling-0}) is more general assumption than Ahlfors regularity (\ref{Ahlfors}). A simple illustrative example could be the case of the\textit{ weighted ball}, $\bB^m:=\big\{x\in\R^m: \|x\|<1\big\}$, of $\bR^m$ with the distance (\ref{distanceB3}) and the weighted measure; \cite{DaiXu}
\begin{equation}\label{def-meas-ball}
d\mu_{\gamma}(x):= (1-\| x\|^2)^{\gamma-1/2} dx, \quad \gamma >-1.
\end{equation}

As in \cite{DaiXu} we have that
\begin{equation}\label{measure-ball}
|B(x, r)| \sim r^m(1-\|x\|^2+r^2)^\gamma,
\end{equation}
which implies that $(\cM, \rho, \mu_{\gamma})$ satisfies the doubling property (\ref{doubling-0}) and the non-collapsing condition (\ref{non-collapsing}).

Precisely, the homogeneous dimension is
$d=m+2\max(\gamma,0)$. Clearly $m\le d$, with the equality exactly when $\gamma=0$, which correspond to the unweighted case and widows the space as an Ahlfors regular one. 

Moreover, it is now apparent how the homogeneous dimension, fundamentally depends on the measure of the space and may or may not be an integer.

Note also that in the weighted case, the proper operator is
\begin{equation*}
L:=L_{\gamma}:= -\sum_{i=1}^m (1-x_i^2)\partial^2_i + 2\sum_{1\le i < j \le m}x_i x_j\partial_i\partial_j
+ (m+2 \gamma)\sum_{i=1}^m x_i \partial_i.
\end{equation*}
which satisfies Assumption II; see \cite{DaiXu,KPX}. 

\vspace{0.2cm}

Note finally that the non-collapsing condition holds true for every space $(\cM,\rho,\mu)$ of homogeneous type which is of finite measure $\mu(\cM)<\infty$; see \cite{CKP}.



\section{Data Illustration}\label{S:data}

In this data illustration, we use earthquake location data for all earthquakes with a reported magnitude of 6.5 or higher between 1990-2021 (inclusive). These data are freely available through the United States Geological Survey website \url{https://earthquake.usgs.gov/earthquakes/search/}. In total, there are $n = 1507$ earthquakes that fit these criteria, and we plot the location of these earthquakes in Figure \ref{fig:locations}. 
\begin{figure}
    \centering
    \includegraphics[width = \textwidth]{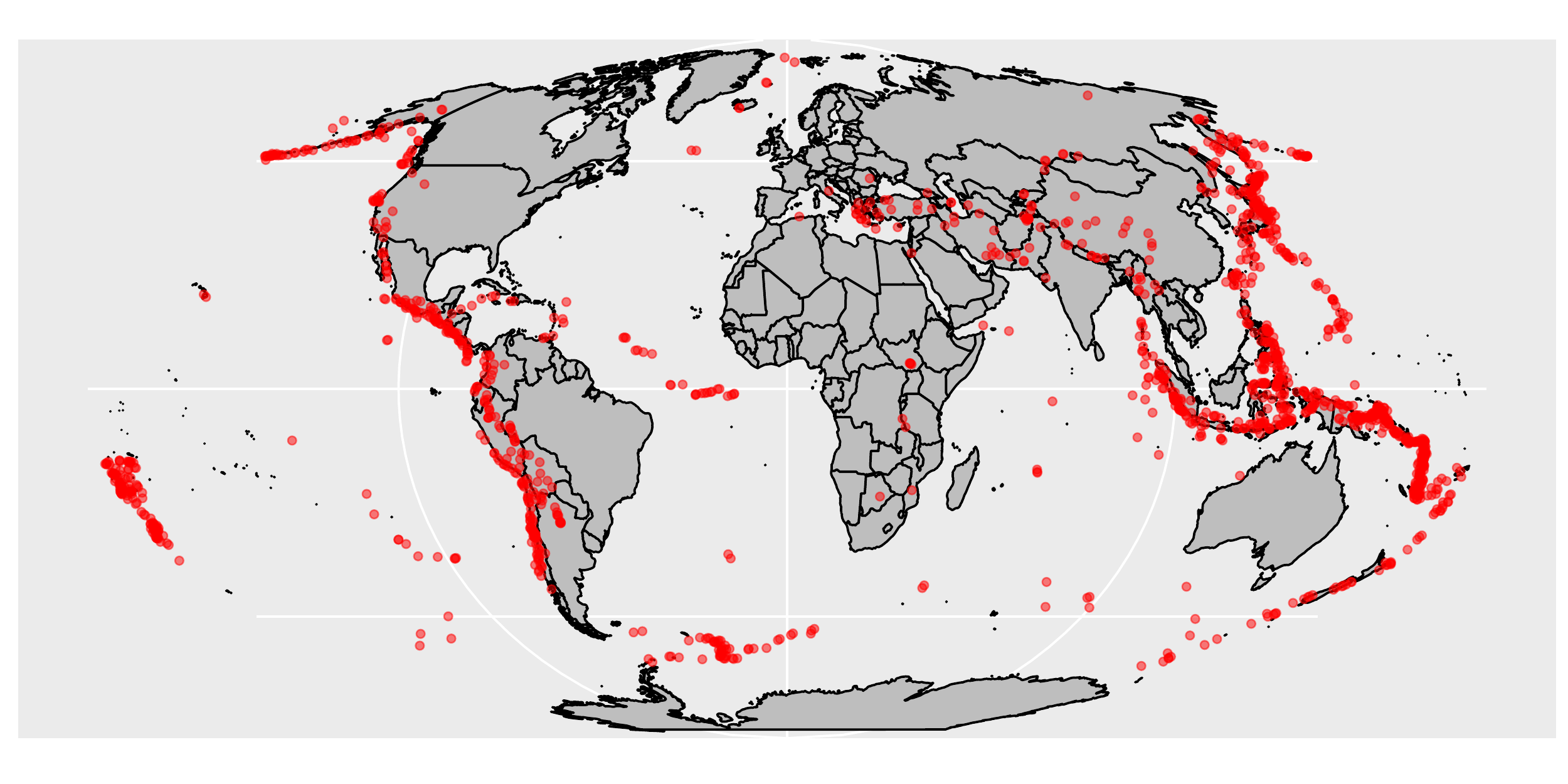}
    \caption{Earthquakes from 1990-2021 with a magnitude of 6.5 or greater.}
    \label{fig:locations}
\end{figure}

To explain some of the earthquake patterns in Figure \ref{fig:locations}, we briefly discuss plate tectonics. The Earth's crust or lithosphere is divided into distinct and irregular sections of solid rock called tectonic plates. The tectonic plates float and gradually move on the molten rock of the Earth's mantle. Many geological events (e.g., volcanic eruptions and earthquakes) occur where different tectonic plates meet. For this reason, high magnitude earthquakes are highly concentrated around tectonic plate boundaries, and this global network of plate boundaries is evident in the earthquake patterns in Figure \ref{fig:locations}. Earthquakes also occur elsewhere in the world at lower rates. 

The Circum-Pacific Belt (the west coasts of the American continents, from Alaska to East Asia, stretching down to the Pacific Islands), sometimes called the Pacific Rim, is the most seismically active. Note that the Pacific Islands (e.g., Tonga, Fiji, New Zealand, and New Caledonia) appear on the left and right of Figure \ref{fig:locations}. The entire Pacific Rim has high concentrations of high magnitude earthquakes, but we point out two other areas with very high concentrations of earthquakes. There are many earthquakes in a small area around the South Sandwich Islands, including earthquakes with 7.5 and 8.1 magnitudes on August 12, 2021. We also point out the Alpide Belt, a region that runs along the Azores, the Mediterranean, the Middle East, the Himalayas, Indonesia, and connects to the Pacific Rim in the Pacific Islands. Given the distribution of earthquakes seen here, we anticipate the need for a heterogeneous density estimate.

These data are distributed globally and are indexed on the sphere. To estimate the density of earthquakes, we approximate the density estimator in \eqref{kdes2} by selecting a finite truncation point $N$,
\begin{equation}\label{eq:trunc_kbed2}
    \hat{f}_{n,h,N}( \xi ) := \frac{1}{n} \sum^n_{i=1} \sum^N_{\nu=0}\frac{1 + 2 \nu}{4 \pi} k(h\sqrt{\nu(\nu +1)}) P_\nu\left( \langle \xi, X_i \rangle \right),
\end{equation}
where $X_i$ are earthquake locations, $k(\cdot)$ is defined as \ref{gsymbols} with $r = 5 + \lceil s\rceil$, and $P_\nu$ are Legendre polynomials.

Because the truncation induces error in the estimation, we anticipate that lower values of $N$ will decrease the accuracy. In Figure \ref{fig:upper_bound}, given $n$, we plot the theoretical upper bound of the truncation error in \eqref{eq:trunc_error_sphere} against the truncation point for various values of $s$. For all values of $s$, the upper bound of truncation error decreases as $N$ increases.
\begin{figure}
    \centering
    \includegraphics[width = .7\textwidth]{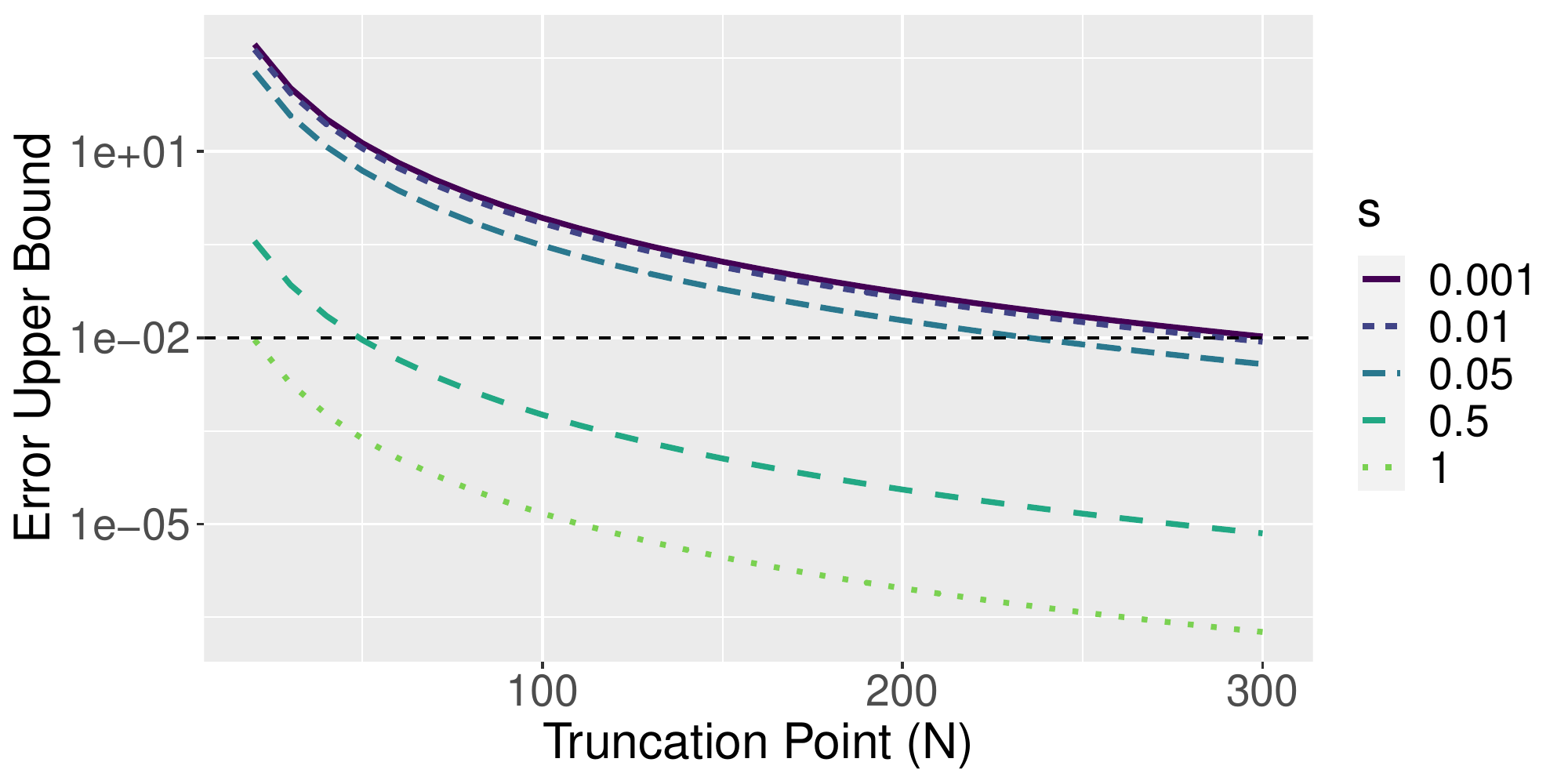}
    \caption{The theoretical upper bound on the truncation error from \eqref{eq:trunc_error_sphere} for combinations of $s$ and $N$. The dashed line indicates an error of 0.01.}
    \label{fig:upper_bound}
\end{figure}

Because \eqref{eq:trunc_kbed2} is not guaranteed to be positive, we use a \emph{rectified} density estimate, 
\begin{equation}\label{eq:rect}
\hat{f}^*_{n,h,N}( \xi ) = \max(10^{-3},\hat{f}_{n,h,N}( \xi ) ).
\end{equation}
In our analysis of these data, we explore the effect of the bandwidth and the truncation point of \eqref{eq:trunc_kbed2} on the density estimates of earthquake locations. We use out-of-sample performance to determine the bandwidth and truncation point. Specifically, we randomly hold out 20\% of the earthquakes as a test dataset $\{X^{\text{test}}_1,...,X^{\text{test}}_{n_{\text{test}}}\}$ to validate the density estimator. Using many different bandwidth and truncation point combinations, we calculate density estimators $\hat{f}^*_{n_{\text{train}},h,N}( \xi )$ using the remaining 80\% of the data. Then, at the hold-out locations, we evaluate $$\hat{f}^*_{n_{\text{train}},h,N}( X^{\text{test}}_1 ),\dots, \hat{f}^*_{n_{\text{train}},h,N}( X^{\text{test}}_{n_{\text{test}}} ).$$ Using these evaluations, we compute the out-of-sample mean log-loss (negative log-score) at the hold out locations 
$$\frac{1}{n_{\text{test}}}\sum^{n_{\text{test}}}_{i=1} -\log\left(\hat{f}^*_{n_{\text{train}},h,N}( X^{\text{test}}_i )\right).$$ The use of log-loss as a proper scoring rule is common; see for example (\cite{good1952}, \cite{gneiting2007strictly}).

 Rather than consider bandwidth directly, we let $h = n^{-1/(2s + 2)}$ and consider $s \in  \{0.001, 0.01, 0.05, 0.5, 1 \}$, where $s$ indexes the smoothness of the density (See Section \ref{sec:holder}). Based on how concentrated earthquake events are, small values of $s$ (i.e., smaller bandwidths) are preferable to smoother alternatives that will yield more uniform density estimators. In our analysis, we also vary the truncation point of the density estimator $N \in  \{5,10,20,30,40,50,75,100\}$. Larger values of $N$ yielded no improvement in log-loss.
We select the truncation point and bandwidth with the lowest out-of-sample mean log-loss. 

We plot the mean log loss as a function of truncation point for various values of $s$ in Figure \ref{fig:log_loss}. Overall, for each $s$, increasing $N$ improves out-of-sample performance up to a point; then, improvement flattens and appears to reach an asymptote. In addition, smaller values of $s$ have better out-of-sample performance; however, values of $s$ less that $0.01$ do not change model performance. The best out-of-sample performance (lowest log loss) is with $N = 75$, and there is no appreciable difference between $s = 0.01$ and $s = 0.001$ (or even smaller values of $s$). For this reason, we use $s = 0.01$.
\begin{figure}
    \centering
    \includegraphics[width = .8\textwidth]{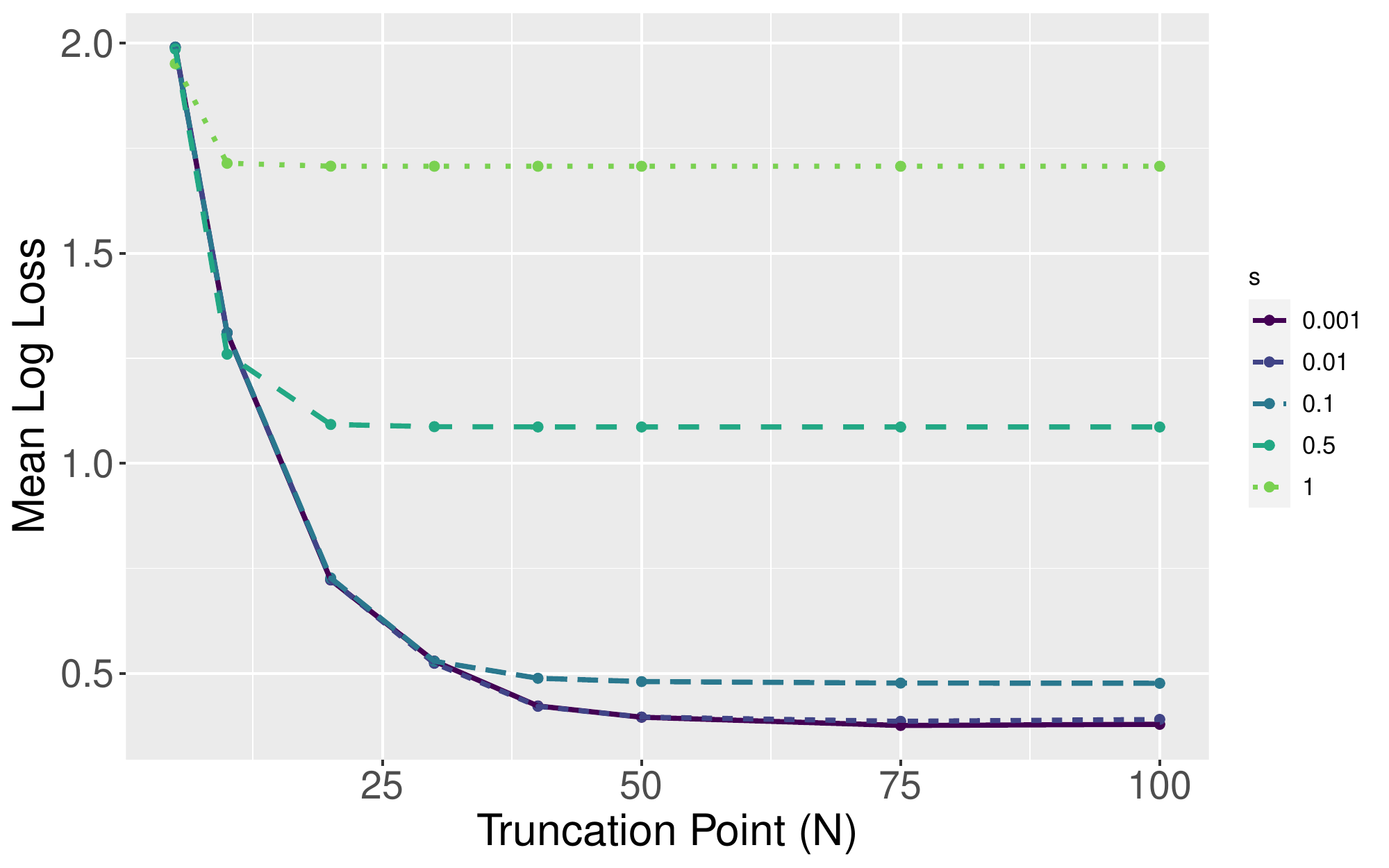}
    \caption{The mean out-of-sample log loss (negative log-score) for various combinations of $s$ and $N$.}
    \label{fig:log_loss}
\end{figure}

For $N = 75$ and $s = 0.01$, we plot a heat map of the estimated density over a fine grid over the Earth in Figure \ref{fig:kde_earthquake}. The colors are on a natural log scale to better see variability in the density. The Pacific Rim is evident in the density estimate, but the most striking features are the high estimated densities around the Pacific Islands and the Pacific Rim. We also note that the South Sandwich Islands and the Alpide belt are visible for their relatively high earthquake densities. 

\begin{figure}
    \centering
    \includegraphics[width = \textwidth]{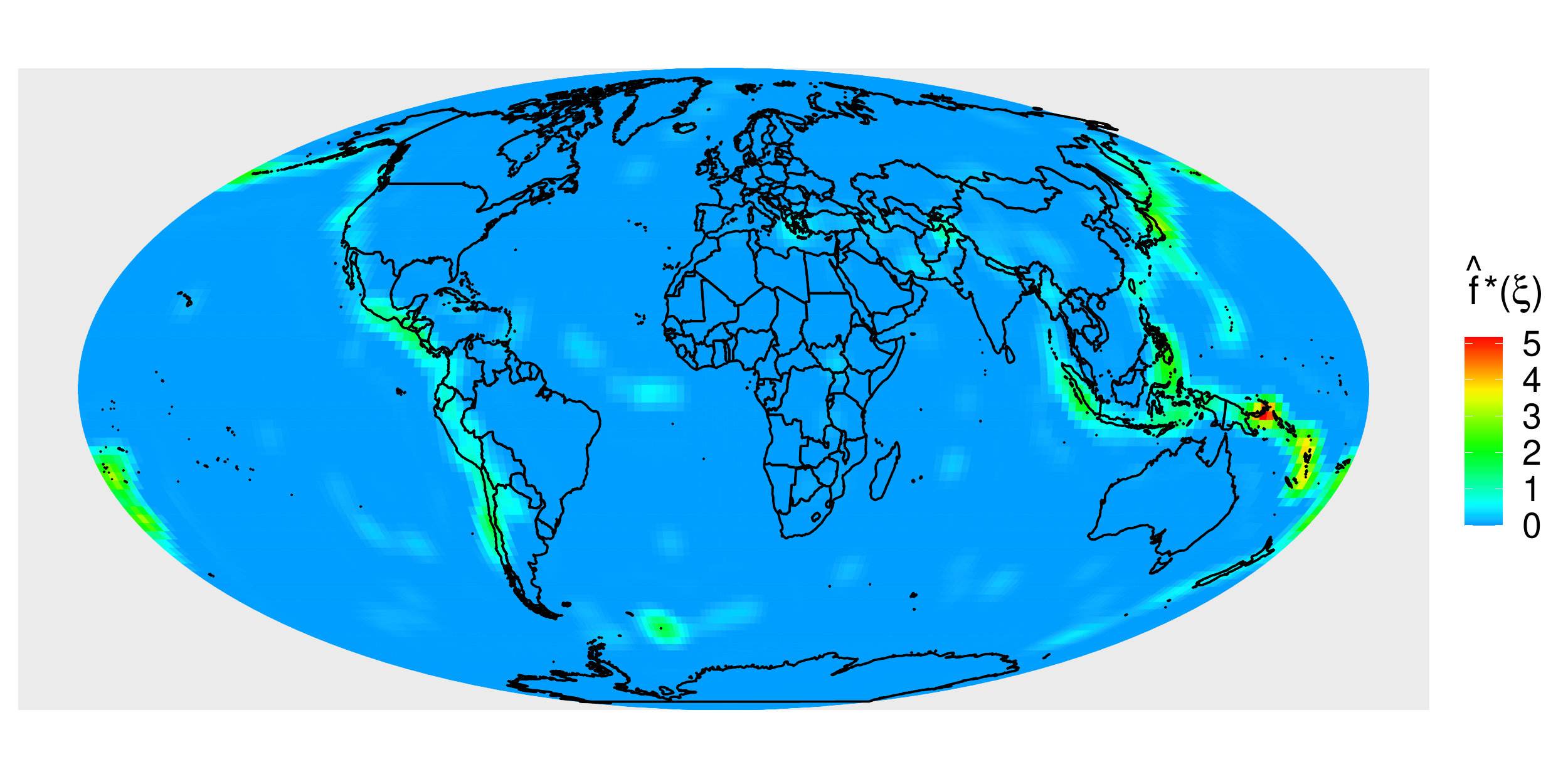}
    \caption{Kernel density estimate for the 1990-2021 earthquake data for $s = 0.01$ and $N = 75$. The colors of the density are presented on the log scale.}
    \label{fig:kde_earthquake}
\end{figure}

In this data illustration, we considered a global dataset of earthquakes from 2021. We compared many truncated kernel density estimators indexed on the sphere on a test set. We found that the kernel density estimates perform better out-of-sample with shorter bandwidths (smaller $s$) and more polynomial terms (higher $N$). For the best combination of $s$ and $N$ considered, we plot the estimated density and comment on its features. In future analyses of these data, one may also estimate the density of earthquake magnitude. In these cases, it may be beneficial to allow the magnitude density to vary smoothly over space as in \cite{sheanshang2021outlier}. In addition, one may account for aftershock excitation in the estimation, as is sometimes used in point process methodology (see, e.g., \cite{hawkes1971point}, \cite{hawkes1971spectra}, \cite{ogata1988statistical}, \cite{ogata1998space}, \cite{white2021generalized}). 

\subsection*{Data availability statement:}
$ $
\vspace{0.2cm}

The data used in this analysis are freely available through the United States Geological Survey website \url{https://earthquake.usgs.gov/earthquakes/search/}


\begin{thebibliography}{3}
\bibitem[Baldi et al.(2009)]{BKMP}
    Baldi, P., Kerkyacharian, G., Marinucci, D., and Picard, D. (2009),
    `Adaptive density estimation for directional data using needlets',
    \textit{The Annals of Statistics}, 37 (6A), 3362--3395.

\bibitem[Baraud et al.(2014)]{Baraud}
    Baraud, Y., Giraud, C. and Huet, S. (2014),
    `Estimator selection in the Gaussian setting',
    \textit{ Ann. Inst. Henri Poincar\'{e} Probab. Stat.}, 50 1092-1119.



\bibitem[Bates and Mio(2014)]{BM}
    Bates, J., Mio, W. (2014),
    `Density estimators of Gaussian type on closed Riemannian manifolds',
    \textit{J. Math. Imaging Vis.}, 50, 827--853.

\bibitem[Berry and Sauer(2017)]{BS}
    Berry, T., Sauer, T. (2017),
    `Density estimation on manifolds with boundary',
    \textit{Comput. Stat. Data Anal.}, 107, 1-17.


\bibitem[Birge(2014)]{Birge}
    Birg\'{e} L. (2014),
    `Model selection for density estimation with $L^2$-loss',
    \textit{Probab. Theory Relat. Fields}, 158, 533-574

\bibitem[Bretagnolle and Huber(1979)]{BH}
    Bretagnolle, J., and Huber, C. (1979),
    `Estimation des densit\'{e}s: risque minimax' (French) 
    \textit{Z. Wahrsch. Verw. Gebiete}, 47 (2), 119-137.

\bibitem[Castillo et al.(2014)]{CaKPi}
    Castillo, I., Kerkyacharian, G., and Picard, D. (2014), 
    `Thomas Bayes' walk on manifolds'. 
    \textit{Probability Theory Related Fields}, 158 (3), 665-710. 

\bibitem[Cleanthous et al.(2020)]{CGKPP}
    Cleanthous, G., Georgiadis, A. G., Kerkyacharian, G., Petrushev, P., and Picard, D. (2020),
    `Kernel and wavelet density estimators on manifolds or more general metric spaces'. 
    \textit{Bernoulli}, 26 (3), 1832-1862.

\bibitem[Cleanthous et al.(2022)]{CGP}
 Cleanthous, G., Georgiadis, A.G., Porcu, E. (2022),
'Oracle inequalities and upper bounds for kernel density estimators on manifolds and more general metric spaces'. \textit{Journal of Nonparametric Statistics},  34 (4), 734-757.

\bibitem[Coulhon et al.(2012)]{CKP}
    Coulhon, T., Kerkyacharian, G., and Petrushev, P. (2012),
    `Heat Kernel Generated Frames in the Setting of Dirichlet Spaces'.	
    \textit{Journal of Fourier Analysis and Applications}, 18 (5), 995--1066.

\bibitem[Dai and Xu(2013)]{DaiXu}
    Dai, F., Xu, Y., (2013),
    'Approximation theory and harmonic analysis on spheres and balls'.
    \textit{Springer Monographs in Mathematics, Springer}.
   

\bibitem[Devroye and Gy\"{o}rfi(1985)]{DG}
    Devroye, L., and Gy\"{o}rfi, L. (1985),
    `Nonparametric Density Estimation: The $L^1$ View'. 
    \textit{Wiley, New York}.
    
\bibitem[Devroye and Lugosi(1996)]{DL} 
   Devroye, L., and Lugosi, L. (1996), 
   `A universally acceptable smoothing factor for kernel density estimation'. 
   \textit{The Annals of Statistics}, 24 , 2499-2512.

\bibitem[Devroye and Lugosi(1997)]{DL2} 
    Devroye, L., and Lugosi, L. (1997), 
    'Nonasymptotic universal smoothing factors, kernel complexity and Yatracos
    classes'. 
    \textit{The Annals of Statistics}, 25, 2626-2637.    
       
\bibitem[Donoho et al.(1996)]{DJKP}  
    Donoho, D. L., Johnstone, I. M., Kerkyacharian, G., and Picard, D. (1996), 
    `Density estimation by wavelet thresholding'.
    \textit{The Annals of Statistics}, 24, 508-539.

\bibitem[Efroimovich(1986)]{EY} 
    Efroimovich, S. Yu. (1986),
    `Non-parametric estimation of the density with unknown smoothness'. 
    \textit{The Annals of Statistics}, 36, 1127-1155.

\bibitem[Georgiadis and Nielsen(2017)]{GN}
    Georgiadis, A. G., and Nielsen, M. (2017), 
    `Pseudodifferential operators on spaces of distributions associated with non-negative self-adjoint operators'. 
    \textit{Journal of Fourier Analysis and Applications}, 23  (2), 344-378.
    
    \bibitem[Gneiting and Raftery, 2007]{gneiting2007strictly}
Gneiting, T. and Raftery, A.~E. (2007).
\newblock Strictly proper scoring rules, prediction, and estimation.
\newblock {\em Journal of the American statistical Association},
  102(477):359--378.

\bibitem[Goldenshluger and Lepski(2014)]{GL}
    Goldenshluger, A., and Lepski, O. (2014),
    `On adaptive minimax density estimation on $\bR^d$'.
    \textit{Probability Theory and Related Fields}, 159, 479-543.  

\bibitem[Goldenshluger and Lepski(2011a)]{GL2}
    Goldenshluger, A., and Lepski, O. (2011a),
    `Uniform bounds for norms of sums of independent random functions'.
    \textit{The Annals of Probability}, 39, 2318-2384.

\bibitem[Goldenshluger and Lepski(2011b)]{GL3} 
    Goldenshluger, A., and Lepski, O. (2011b),
    `Bandwidth selection in kerrnel density estimation: oracle inequalities and adaptive minimax optimality'. 
    \textit{The Annals of Statistics}, 39, 1608-1632.

\bibitem[Goldenshluger and Lepski(2022a)]{GLne1} 
Goldenshluger, A. and Lepski, O. (2022), `Minimax estimation of norms of a probability
density: I. Lower bounds'. Bernoulli, 28 (2), 1120-1154.


\bibitem[Goldenshluger and Lepski(2022b)]{GLne2}
Goldenshluger, A. and Lepski, O.  (2022), `Minimax estimation of norms of a probability density: II. Rate-optimal estimation procedures'. Bernoulli, 28 (2), 1155-1178.



\bibitem[Good, 1952]{good1952}
Good, I. (1952).
\newblock Rational decisions.
\newblock {\em Journal of the Royal Statistical Society: Series B
  (Methodological)}, 14(1):107--114.

\bibitem[H\"{a}rdle et al.(1998)]{HGPT}
    H\"{a}rdle, W., Kerkyacharian, G., Picard, D., and Tsybakov, A.B. (1998),
    `Wavelets, approximation, and statistical applications'. 
    \textit{Lecture Notes in Statistics}, 129. Springer-Verlag, New York.


\bibitem[Hall et al.(1987)]{HWC}
    Hall, P., Watson, G. S., Cabrera, J.. (1987),
    `Kernel density estimation with spherical data'. 
    \textit{Biometrika}, 74, 751-762.


\bibitem[Hasminskii and Ibragimov(1990)]{HI}  
    Hasminskii, R.  and Ibragimov, I.A. (1990),
    `On density estimation in the view of Kolmogorov's ideas in approximation
theory'. 
    \textit{The Annals of Statistics}, 18, 999-1010.

\bibitem[Hawkes, 1971a]{hawkes1971point}
Hawkes, A.~G. (1971a).
\newblock Point spectra of some mutually exciting point processes.
\newblock {\em Journal of the Royal Statistical Society: Series B
  (Methodological)}, 33(3):438--443.

\bibitem[Hawkes, 1971b]{hawkes1971spectra}
Hawkes, A.~G. (1971b).
\newblock Spectra of some self-exciting and mutually exciting point processes.
\newblock {\em Biometrika}, 58(1):83--90.


\bibitem[Ibragimov and Khasminski(1980)]{IK}  
    Ibragimov, I.A., and Khasminski, R.Z. (1980),
    `An estimate of the density of a distribution'. 
    Zap. Nauchn. Sem. Leningrad. Otdel. Mat. Inst. Steklov. (LOMI) 98, 61-85.

\bibitem[Juditsky and Lambert-Lacroix(2004)]{Lacroix-Lambert}
    Juditsky, A., and Lambert-Lacroix, S. (2004),
    `On minimax density estimation on $\R$',
    \textit{Bernoulli}, 10(2), 187--220.
  
\bibitem[Kerkyacharian et al.(2018)]{KOPP}
    Kerkyacharian, G., Ogawa, S., Petrushev, P., and Picard, D. (2018),
    `Regularity of Gaussian processes on Dirichlet spaces',
    \textit{Constructive Approximation}, 47(2), 277--320.

\bibitem[Kerkyacharian and Petrushev(2015)]{KP}
    Kerkyacharian, G., and Petrushev, P. (2015),
    `Heat kernel based decomposition of spaces of distributions in the framework of Dirichlet spaces',
    \textit{Transactions of the American Mathematical Society}, 367, 121--189.

\bibitem[Kerkyacharian et al.(2020)] {KPX}
    Kerkyacharian, G., Petrushev, P., and Xu, Y. (2020),
    `Gaussian bounds for the weighted heat kernels on the interval, ball and simplex',
    \textit{Constructive Approximation}, 51, 73--122.

 
  \bibitem[Kerkyacharian et al.(1996)]{KePT}
    Kerkyacharian, G., Picard, D., and Tribouley, K. (1996),
    `$L^p$ adaptive density estimation', 
    \textit{Bernoulli}, 2, 229--247.

\bibitem[Kerkyacharian et al.(2001)]{KeLP} 
    Kerkyacharian, G., Lepski, O., and Picard, D. (2001), 
    `Nonlinear estimation in anisotropic multi-index denoising',
    \textit{Probability Theory and Related Fields}, 121, 137--170.

\bibitem[Kerkyacharian et al.(2008)]{KLP2} 
    Kerkyacharian, G., Lepski, O., and Picard, D. (2008), 
    `Nonlinear estimation in anisotropic multiindex denoising. Sparse case', 
    \textit{Theory of Probability and its Applications}, 52, 58--77.

\bibitem[Kyriazis et al. (2008)]{KPY}
   Kyriazis, G., Petrushev, P., Xu, Y.  (2008), 
'Decomposition of weighted Triebel-Lizorkin and Besov spaces on the ball'.
\textit{Proc. Lond. Math. Soc.}, (3) 97, No. 2, 477-513.

  
  
\bibitem[Massart(2007)]{Massart}  
 Massart, P. (2007). `Concentration Inequalities and Model Selection'. Lecture Notes in Math. 1896.
Springer, Berlin. 
  
  
\bibitem[Ogata, 1988]{ogata1988statistical}
Ogata, Y. (1988).
\newblock Statistical models for earthquake occurrences and residual analysis
  for point processes.
\newblock {\em Journal of the American Statistical Association}, 83(401):9--27.

\bibitem[Ogata, 1998]{ogata1998space}
Ogata, Y. (1998).
\newblock Space-time point-process models for earthquake occurrences.
\newblock {\em Annals of the Institute of Statistical Mathematics},
  50(2):379--402.
  
\bibitem[Parzen (1962)]{Parzen}
Parzen, E. (1962), `On the estimation of a probability density function and
mode'. \textit{Annals of Mathematical Statistics}, 33, 1065-1076.
 
\bibitem[Pelletier(2005)]{Pel1}  
    Pelletier, B. (2005),
    `Kernel density estimation on Riemannian manifolds',
    \textit{Statistics and probability letters}, 73(3), 297--304.

\bibitem[Pelletier(2006)]{Pel2} 
    Pelletier, B. (2006),
    `Non-parametric regression estimation on closed Riemannian manifolds',
    \textit{Journal of Nonparametric Statistics}, 18(1), 57--67.


\bibitem[Prugove\v{c}ki(1981)]{Prugov}
    E. Prugove\v{c}ki,
    Quantum mechanics in Hilbert space. Second edition. Pure and Applied Mathematics, 92.
    Academic Press, Inc., New York -- London, 1981.

\bibitem[Reed and Simon(1980)]{RS}
    M. Reed, B. Simon,
    Methods of modern mathematical physiscs I: Functional analysis,
    Academic Press, New York, 1980.




\bibitem[Rigollet(2006)]{Rig} 
   Rigollet, Ph. (2006),
   `Adaptive density estimation using the blockwise Stein method',
   \textit{Bernoulli}, 12, 351--370.


\bibitem[Rigollet and Tsybakov(2007)]{RigT}  
  Rigollet, Ph., and Tsybakov, A.B. (2007),
  `Linear and convex aggregation of density estimators',
  \textit{Mathematical Methods of Statistics}, 16, 260--280.

\bibitem[Rosenblatt (1956)]{Ros}
Rosenblatt, M. (1956), `Remarks on some nonparametric estimates of a density
function', \textit{Annals of Mathematical Statistics}, 27, 832-837.

\bibitem[Samarov and Tsybakov(2007)]{ST}  
    Samarov, A., and Tsybakov, A.B. (2007),
    `Aggregation of density estimators and dimension reduction', 
    Advances in Statistical Modeling and Inference, pp. 233-251, Ser. Biostat., Vol. 3. World Sci. Publ., Hackensack (2007).

\bibitem[Sandryhaila and Moura(2013)]{Sand-moura}
    Sandryhaila, A., and Moura, J.M.F. (2013),
    `Discrete signal processing on graphs',
    \textit{IEEE Transactions on Signal Processing}, 61(7), 1644--1656.
    
\bibitem[Sheanshang et~al., 2021]{sheanshang2021outlier}
Sheanshang, D.~M., White, P.~A., and Keeler, D.~G. (2021).
\newblock Outlier accommodation with semiparametric density processes: A study
  of {A}ntarctic snow density modelling.
\newblock {\em Statistical Modelling}, page 1471082X211043946.

\bibitem[Silverman(1986)]{silverman}
    Silverman, B.W. (1986),
    \textit{Density estimation for statistics and data analysis},
    London: Monographs on Statistics and Applied Probability, Chapman \& Hall.

\bibitem[Starck et al.(2010)]{starck2010sparse}
    Starck, J.-L., Murtagh, F., and Fadili, J.M. (2010),
    \textit{Sparse image and signal processing: wavelets, curvelets, morphological diversity},
    Cambridge: Cambridge University Press.



\bibitem[Triebel (1983)]{T}
    Triebel, H.,
    Theory of function spaces,
    Monographs in Math. Vol. 78, Birkh\"{a}user, Verlag, Basel, 1983.


\bibitem[Tsybakov(2009)]{Tsybakov}
    Tsybakov, A.B. (V. Zaiats, trans.) (2009),
    \textit{Introduction to nonparametric estimation,},
    Springer Series in Statistics,
    New York: Springer
    
    
\bibitem[White and Gelfand, 2021]{white2021generalized}
White, P.~A. and Gelfand, A.~E. (2021).
\newblock Generalized evolutionary point processes: Model specifications and
  model comparison.
\newblock {\em Methodology and Computing in Applied Probability},
  23(3):1001--1021.

\bibitem[Yoshida(1978)]{Yosida}
    Yoshida, K. (1978), 
    \textit{Functional analysis}, 
    Berlin: Springer-Verlag.

\end{thebibliography}
\end{document}